\documentclass[a4paper,11pt,reqno]{amsart}

\usepackage[utf8]{inputenc}
\usepackage{comment}
\usepackage[english]{babel}
\usepackage{amsmath, mathtools}
\usepackage{amssymb}
\usepackage{amsfonts}
\usepackage{amsthm}
\usepackage{bbm}
\usepackage{enumitem}
\usepackage{xcolor}
\usepackage[colorlinks=false]{hyperref}

\def\d{\mathrm{d}}

 \providecommand{\Xint}[1]{\mathchoice
    {\XXint\displaystyle\textstyle{#1}}%
    {\XXint\textstyle\scriptstyle{#1}}%
    {\XXint\scriptstyle\scriptscriptstyle{#1}}%
    {\XXint\scriptscriptstyle\scriptscriptstyle{#1}}%
    \!\int}
  \providecommand{\XXint}[3]{{\setbox0=\hbox{$#1{#2#3}{\int}$}
      \vcenter{\hbox{$#2#3$}}\kern-.5\wd0}}
  
  \providecommand{\dashint}{\mathop{\Xint-}}

\def\ristretto{\lfloor}

\usepackage{ esint }

\newtheorem{theorem}{Theorem}[section]

\newtheorem{lemma}[theorem]{Lemma}
\newtheorem{proposition}[theorem]{Proposition}

\theoremstyle{definition}
\newtheorem{definition}[theorem]{Definition}
\newtheorem{remark}[theorem]{Remark}

\numberwithin{equation}{section}

\newcommand{\R}{\mathbb{R}}
\newcommand{\N}{\mathbb{N}}

\newcommand{\ssubset}{\subset\joinrel\subset}

\newcommand{\Rd}{{\R}^d}
\newcommand{\Sd}{{\mathbb{S}}^{d-1}}

\newcommand{\Ld}{{\mathcal{L}}^d}

\def\Xint#1{\mathchoice 
  {\XXint\displaystyle\textstyle{#1}}%
  {\XXint\textstyle\scriptstyle{#1}}%
  {\XXint\scriptstyle\scriptscriptstyle{#1}}%
  {\XXint\scriptscriptstyle\scriptscriptstyle{#1}}%
  \!\int} 
\def\XXint#1#2#3{{\setbox0=\hbox{$#1{#2#3}{\int}$} 
  \vcenter{\hbox{$#2#3$}}\kern-.5\wd0}} 
\def\-int{\Xint -}

\usepackage{latexsym,amsfonts}
\usepackage{mathrsfs}
\usepackage{centernot}

\numberwithin{equation}{section}

\begin{document}

\title[Free-discontinuity problems  with $p(\cdot)$-growth]{Regularity of minimizers for  free-discontinuity problems  with $p(\cdot)$-growth}

\author{Chiara Leone}
\address[Chiara Leone]{Department of Mathematics and Applications ``R. Caccioppoli'', University of Naples Federico II, Via Cintia, Monte S. Angelo, 80126 Naples, Italy}
\email{chiara.leone@unina.it}
 
\author{Giovanni Scilla}
\address[Giovanni Scilla]{Department of Mathematics and Applications ``R. Caccioppoli'', University of Naples Federico II, Via Cintia, Monte S. Angelo, 80126 Naples, Italy}
\email[Giovanni Scilla]{giovanni.scilla@unina.it}

\author{Francesco Solombrino}
\address[Francesco Solombrino]{Department of Mathematics and Applications ``R. Caccioppoli'', University of Naples Federico II, Via Cintia, Monte S. Angelo, 80126 Naples, Italy}
\email{francesco.solombrino@unina.it}

\author{Anna Verde}
\address[Anna Verde]{Department of Mathematics and Applications ``R. Caccioppoli'', University of Naples Federico II, Via Cintia, Monte S. Angelo, 80126 Naples, Italy}
\email{anna.verde@unina.it}



\begin{abstract}

A regularity result for free-discontinuity energies defined on the space $SBV^{p(\cdot)}$ of special functions of bounded variation with variable exponent is proved, under the assumption of a log-H\"older continuity for the variable exponent $p(x)$. Our analysis expand on the regularity theory for minimizers of a class of free-discontinuity problems in the nonstandard growth case. This may be seen as a follow-up of the paper \cite{FMT}, dealing with a constant exponent. 
\end{abstract}

\maketitle

\tableofcontents

\section{Introduction}
Integral functionals with non-standard growth, introduced  by Zhikov~\cite{zikov, zikov3}, are customary in the modeling of composite materials which exhibit a strongly anisotropic behavior. In such a setting, a point-dependent integrability of the deformation gradient is usually assumed, which may be captured, for instance, in terms of variable exponents spaces (see ~\cite{KR, SH}). Over the years, the regularity properties of minimizers (in the Sobolev space $W^{1, p(\cdot)}(\Omega; \R^m)$, where $\Omega$ is a reference configuration) of variational integrals of the form
\begin{equation}\label{introeq:bulk}
\int_\Omega f(x,\nabla u(x))\,\mathrm{d}x\,,
\end{equation}
 under  a $p(x)$-growth condition
\begin{equation}\label{introeq:growth}
c|\xi|^{p(x)} \le f(x, \xi)\le C(1+|\xi|^{p(x)})\,,
\end{equation}
has been the subject of many contributions. Among them, we may mention ~\cite{AM, AM2, CosM, MR3209686,  ELEUTERI2023103815, HastoOk, Z2}. The common key assumptions to these papers are superlinearity of $p(\cdot)$ (meaning $\min_{\Omega} p(\cdot) >1$), and the (possibly strong) log-H\"older continuity of the exponent. This condition on the modulus of continuity of the variable exponent was firstly considered by {Zhikov} in \cite{Z2} to prevent the Lavrentiev phenomenon. Roughly speaking,  it allows one to freeze the exponent on small balls around a point, as pointed out in \cite[Lemma~3.2]{Dieni} and is particularly suitable for blow-up methods. Besides regularity issues, we may point out for instance its usage in \cite{ABF}, in order to show that the singular part of the measure representation of relaxed functionals with growth \eqref{introeq:growth} disappears. 

The focus of our paper is, instead, on the regularity of minimizers of {\it free-discontinuity functionals} in variable exponent spaces. In such a setting, which is rather natural to describe failure phenomena such as fracture and damage, singularities may appear in the form of jump discontinuities.  These problems are characterized  by the competition between a ``bulk'' energy of the form \eqref{introeq:bulk} and a ``surface" energy accounting for the energy spent to produce a crack (\cite{Griffith:1921, FrancfortMarigo:1998}). A prototypical functional is then taking the form
\begin{equation}\label{introeq:generalprob}
\int_\Omega f(x,\nabla u(x))\,\mathrm{d}x+\int_{J_u} g\big(x,  u^+(x), u^-(x), \nu_u(x)\big) \, {\rm d} \mathcal{H}^{d-1} (x)\,.
\end{equation}
Above, $J_u$ is the set of jump discontinuities of $u$ with normal $\nu_u$, which, exactly like the gradient $\nabla u$ and the one-sided traces $u^+(x)$, $u^-(x)$, have in general to be understood in an approximate measure-theoretical sense  (see Section \ref{s:bv}). The $p(\cdot)$-growth condition \eqref{introeq:growth} is assumed on $f$, while $g$ is bounded from above and from below by positive constants. 

The above problem is usually complemented with lower order fidelity terms or boundary data, which, whenever $p(\cdot)$ is superlinear,  allow one to apply the results of \cite{A89, Amb} (see also \cite{Fr} for the case of boundary data) and obtain sequential coercivity of the functional in the space of Special functions of Bounded Variation ($SBV$), see \cite{Ambrosio-Fusco-Pallara:2000}. Under the $BV$-ellipticity of $g$, which provides lower semicontinuity of the surface integral, the well-posedness of the minimum problem \eqref{introeq:generalprob}  in the subspace $SBV^{p(\cdot)}$ of $SBV$ functions with $p(\cdot)$-integrable gradients can be then inferred from the results of \cite{DeG}, whenever $f$ is convex in the gradient variable, or more in general of \cite{DCLV}, whenever $f$ is quasiconvex and and the exponent is log-H\"older continuous. We also refer the reader to the recent \cite{ars}, where mere continuity of $p$ is shown to be sufficient for lower semicontinuity of the bulk energy. We however warn the reader that  log-H\"older continuity is again going to play a central role when coming to regularity issues.
It is also worth mentioning that, besides Materials Science, applications of a variable exponent in the setting of functions of bounded variation already appeared in image reconstruction~\cite{CLR, Hasto, Hasto2, Li}. This is the setting where free-discontinuity problems were originally introduced \cite{MS}.

Now, for $g\equiv c$, with $c>0$, \eqref{introeq:generalprob} can be seen as a weak formulation of the problem
\begin{equation}\label{introeq:strongprob}
\int_{\Omega\setminus K} f(x,\nabla u(x))\,\mathrm{d}x+c\mathcal{H}^{d-1} (K)
\end{equation}
for a  {\it closed} set $K$, not prescribed a-priori, and a deformation $u$ which is smooth outside of $K$. One is then willing to show that, given a minimizer $u$ of \eqref{introeq:generalprob}, the pair $(u, \overline{J}_u)$ indeed provides a minimizer of \eqref{introeq:strongprob}. This is clearly the case whenever $\mathcal{H}^{d-1}(\overline{J}_u\setminus J_u)=0$; notice that, once this is achieved, setting $K=\overline{J}_u$ the smoothness of $u$ in $\Omega \setminus K$ can be then deduced from the regularity theory for variational integrals in Sobolev spaces. For the model case 
$$f(x,\nabla u(x))=|\nabla u(x)|^p,$$
with $p$ a constant exponent, this difficult task was first accomplished in the seminal paper \cite{DeGCarLea}. There, it was namely shown that  $\mathcal{H}^{d-1}(\overline{S}_u\setminus S_u)=0$, where $S_u$ is the singular set of $u$, a superset of $J_u$ which, in a $BV$ setting, differs therefrom only by a $\mathcal{H}^{d-1}$-null set. A crucial estimate in order to get their result was the so-called density lower bound 
\[
\mathcal{H}^{d-1}(S_u\cap B_\rho(x_0)) \ge \theta_0 \rho^{d-1}
\]
for $x_0\in S_u$ and sufficiently small balls $B_\rho(x_0)$, with $\theta_0$ independent of $x_0$ and $\rho$. This requires a fine decay analysis for the energy on small balls, which is partially simplified by the homogeneity of the bulk energy, ensuring that minimality is invariant under suitable rescaling.

If one renounces to homogeneity, regularity results for minimizers of  a class of nonhomogenous bulk integrands with $p$-growth in \eqref{introeq:generalprob} have been obtained in \cite{FMT}, whose results we extend to the variable growth setting.

\noindent \emph{Description of our results.} In this paper, we focus on the scalar-valued case $u\colon \Omega \to \mathbb{R}$ and consider $f(x, \xi)= |\xi|^{p(x)}+h(x,\xi)$ in \eqref{introeq:generalprob}, where $h$ is a continuous function, convex in $\xi$ and has $p(x)$-growth. We assume that the variable exponent is strongly log-H\"older continuous (see \eqref{eq:rinfologhold}) and show that minimizers of  the weak formulation \eqref{introeq:generalprob} are strong minimizers in the sense  clarified above.

The proof of the crucial density lower bound goes, exactly as in  \cite{DeGCarLea, FMT}, through a decay Lemma (see Lemma \ref{lem:decay}). One assumes by contradiction that the energy is decaying faster than $\rho^{d-1}$ around a jump point $x_0$. Setting $\bar p=p(x_0)$, one may exploit smallness of the energy to show that a scaled copy of blown-up sequences converges to a smooth minimizer of a variational integral of the type
\[
\int_{B_1} (|\nabla u|^{\bar p}+ h_\infty(\nabla u))\,\mathrm{d}x
\]
for which decay estimates independent of $\bar p$ are available by \cite{AM2}, and provide a contradiction. The function $h_\infty$ is recovered as locally uniform limit of a scaled version of $h$ acting on a scaled deformation gradient. The strong log-H\"older continuity assumption is crucial for a proper choice of the scaling constants, despite the presence of a variable exponent, in order to ensure minimality of the limit function. This also requires the proof of a $\Gamma$-liminf type inequality (see Step 1 of Theorem \ref{gammaconv}) for sequences of functionals with variable growth, which is achieved by means of the Lusin-type approximation in \cite[Theorem~3.1]{DCLV}, recalled in Theorem \ref{thm:dclvLip}. With the main tools of  Theorem \ref{gammaconv} and Lemma \ref{lem:decay} at our disposal, we can prove the density lower bound and eventually recover existence of strong minimizers in Theorem \ref{thm:main}.

\noindent \emph{Outline of the paper.} The paper is structured as follows. {In Section~\ref{sec: prel} we fix the basic notation and recall some definitions in variable exponent spaces (Subsection~\ref{sec:variableexp}) and in $SBV$ (Subsection~\ref{s:bv}), while Subsection~\ref{sec:gsbvp(x)} deals with the space $SBV^{p(\cdot)}$, the Poincar\'e inequality and some of its consequences useful in the sequel. In Subsection~\ref{sec:lusinsbpx} we recall a Lusin-type approximation result in $SBV^{p(\cdot)}$, while Subsection~\ref{sec:freepx} contains some definitions and results for free-discontinuity problems in the variable exponent setting. Section~\ref{sec:smalljump} is entirely devoted to the proof of a technical tool concerning the asymptotic behavior of almost minimizers with small jump sets. The main result of the paper is contained in Section~\ref{sec:strongmin}: in Subsection~\ref{sec:mainassump} we state the problem and list the main assumptions on the energies, in Subsection~\ref{sec:decay} we establish a crucial decay estimate for our functionals while in Subsection~\ref{sec:densitylowerb} we prove a density lower bound and then the main result with Theorem~\ref{thm:mainthm}.}

\section{Basic notation and preliminaries}\label{sec: prel}

 We start with some basic notation.   Let $\Omega \subset \R^d$  be  open and bounded.
 For every $x\in \Rd$ and $r>0$ we indicate by $B_r(x) \subset \Rd$ the open ball with center $x$ and radius $r$. If $x=0$, we will often use the shorthand $B_r$.  For $x$, $y\in \Rd$, we use the notation $x\cdot y$ for the scalar product and $|x|$ for the  Euclidean  norm.   Moreover, we let  $\Sd:=\{x \in \Rd \colon |x|=1\}$ and we denote by $\R^d_0$ the set $\R^d\setminus\{0\}$. The $m$-dimensional Lebesgue measure of the unit ball in $\R^m$ is indicated by $\gamma_m$ for every $m \in \N$.   We denote by $\Ld$ and $\mathcal{H}^k$ the $d$-dimensional Lebesgue measure and the $k$-dimensional Hausdorff measure, respectively.  
The closure of $A$ is denoted by $\overline{A}$. The diameter of $A$ is indicated by ${\rm diam}(A)$.  
We write $\chi_A$ for the  characteristic  function of any $A\subset  \R^d$, which is 1 on $A$ and 0 otherwise.  If $A$ is a set of finite perimeter, we denote its essential boundary by $\partial^* A$,   see  \cite[Definition 3.60]{Ambrosio-Fusco-Pallara:2000}.   

\subsection{Variable exponent Lebesgue spaces}\label{sec:variableexp}

We briefly recall the notions of variable exponents and variable exponent Lebesgue spaces. We refer the reader to \cite{DHHR} for a comprehensive treatment of the topic. 

A measurable function $p:\Omega\to[1,+\infty)$ will be called a \emph{variable exponent}. Correspondingly, for every $A\subset\Omega$ we define
\begin{equation*}
p^+_A:=\mathop{{\rm ess}\,\sup}_{x\in A} p(x)\,\mbox{\,\, and \,\,}p^-_A:=\mathop{{\rm ess}\,\inf}_{x\in A} p(x)\,,
\end{equation*}
while $p^+_\Omega$ and $p^-_\Omega$ will be denoted by $p^+$ and $p^-$, respectively.

For a measurable function $u:\Omega\to\R$ we define the \emph{modular} as
\begin{equation*}
\varrho_{p(\cdot)}(u):=\int_\Omega|u(x)|^{p(x)}\,\mathrm{d}x
\end{equation*}
and the (Luxembourg) \emph{norm}
\begin{equation*}
\|u\|_{L^{p(\cdot)}(\Omega)}:=\inf\{\lambda>0:\,\, \varrho_{p(\cdot)}(u/\lambda)\leq1\}\,.
\end{equation*}
The \emph{variable exponent Lebesgue space} $L^{p(\cdot)}(\Omega)$ is defined as the set of measurable functions $u$ such that $\varrho_{p(\cdot)}(u/\lambda)<+\infty$ for some $\lambda>0$. In the case $p^+<+\infty$, $L^{p(\cdot)}(\Omega)$ coincides with the set of functions such that $\varrho_{p(\cdot)}(u)$ is finite. It can be checked that $\|\cdot\|_{L^{p(\cdot)}(\Omega)}$ is a norm on $L^{p(\cdot)}(\Omega)$. Moreover, if $p^+<+\infty$, it holds that
\begin{equation}
\varrho_{p(\cdot)}(u)^\frac{1}{p^+} \leq \|u\|_{L^{p(\cdot)}(\Omega)} \leq \varrho_{p(\cdot)}(u)^\frac{1}{p^-}
\label{eq:normineq}
\end{equation}
if $\|u\|_{L^{p(\cdot)}(\Omega)}>1$, while an analogous inequality holds by exchanging the role of $p^-$ and $p^+$ if $0\leq \|u\|_{L^{p(\cdot)}(\Omega)}\leq 1$. Another useful property of the modular, in the case $p^+<+\infty$, is the following one: 
\begin{equation}
\min\{\lambda^{p^+}, \lambda^{p^-}\}\varrho_{p(\cdot)}(u) \leq \varrho_{p(\cdot)}(\lambda u) \leq \max\{\lambda^{p^+}, \lambda^{p^-}\}\varrho_{p(\cdot)}(u)
\label{eq:normineq2}
\end{equation}
for all $\lambda>0$.

We say that a function $p:\Omega\to\R$ is \emph{log-H\"older continuous} on $\Omega$ if the modulus of continuity for $p(x)$ satisfies
\[
 \omega(|x-y|)\leq \frac{C}{-\log |x-y|}\,,\quad \forall x,y\in\Omega\,,\, |x-y|\leq\frac{1}{2}\,.
\]
with $C$ a positive constant. In other words 
\begin{equation}\label{eq:loghold}
\limsup_{\rho\to 0}\omega(\rho)\log\left(\frac{1}{\rho}\right)<+\infty.
\end{equation}
To prove our regularity result the previous condition {will be reinforced into the \emph{strong log-H\"older continuity}}
\begin{equation}\label{eq:rinfologhold}
\limsup_{\rho\to 0}\omega(\rho)\log\left(\frac{1}{\rho}\right)=0,
\end{equation}
in complete accordance with the theory of regularity in the variable Sobolev framework (see \cite{AM2}).


The following lemma provides an extension to the variable exponent setting of the well-known embedding property of classical Lebesgue spaces (see, e.g., \cite[Corollary~3.3.4]{DHHR}).

\begin{lemma}\label{embedding}
Let $p,q$ be measurable variable exponents on $\Omega$, and assume that $\mathcal{L}^d(\Omega)<+\infty$. Then $L^{p(\cdot)}(\Omega)\hookrightarrow L^{q(\cdot)}(\Omega)$ if and only if $q(x)\leq p(x)$ for $\mathcal{L}^d$-a.e. $x$ in $\Omega$. The embedding constant is less or equal to $2(1+\mathcal{L}^d(\Omega))$ and $2\max\{\mathcal{L}^d(\Omega)^{(\frac{1}{q}-\frac{1}{p})^+, (\frac{1}{q}-\frac{1}{p})^-}\}$.
\label{lem:embed}
\end{lemma}


\subsection{$BV$ and $SBV$ functions}\label{s:bv}

For a general survey on the spaces of $BV$ and $SBV$ functions 
we refer for instance to \cite{Ambrosio-Fusco-Pallara:2000}. Below, we just recall some basic definitions useful in the sequel. 

If  $u\in L^1_{\rm loc}(\Omega)$ and $x\in\Omega$, the {\it precise representative of $u$ at $x$} is defined
as the unique value $\widetilde{u}(x)\in\R$ such that
$$
\lim_{\rho\to 0^+} \frac{1}{\rho^d}\int_{{B_\rho(x)}}\!|u(y)-\widetilde{u}(x)|\,\d x=0\,.
$$
The set of points in $\Omega$ where the precise representative of $x$ is not defined is called the
{\it approximate singular set} of $u$ and denoted by $S_u$. We say that a point $x\in\Omega$ is an approximate jump point of $u$
if there exist $a,b\in\R$ and $\nu\in\mathbb{S}^{d-1}$, such that $a\not = b$ and
$$
\lim_{\rho\to 0^+}\dashint_{B^+_\rho(x,\nu)} |u(y)-a|\, \d y=0
\qquad{\rm and}\qquad
\lim_{\rho\to 0^+}\dashint_{B^-_\rho(x,\nu)} |u(y)-b|\, \d y=0
$$
where $B^\pm_\rho(x,\nu):= \{y\in B_\rho(x)\ :\ \langle y-x,\nu\rangle\gtrless0\}$.
The triplet $(a,b,\nu)$ is uniquely determined by the previous formulas, up to a permutation
of $a,b$ and a change of sign of $\nu$, and it is denoted by $(u^+(x),u^-(x),\nu_u(x))$.
The Borel functions $u^+$ and $u^-$  are called the {\it upper and
lower approximate limit} of $u$ at the point $x\in\Omega$. The set of approximate jump points of $u$ is denoted by $J_u\subseteq S_u$.

The space ${BV}(\Omega)$ of {\it functions of bounded variation} is defined as the set
of all $u\in L^1(\Omega)$ whose distributional gradient $Du$ is a bounded Radon measure on $\Omega$
with values in $\R^d$.  Moreover, the usual decomposition
\begin{equation*}
Du = \nabla u\,\mathcal{L}^d + D^c u + (u^+-u^-)\otimes \nu_u\,\mathcal{H}^{d-1}\ristretto{J_u}
\end{equation*}
\noindent
holds, where $\nabla u$ is the Radon-Nikod\'ym derivative of $Du$ with respect to the Lebesgue measure and $D^cu$ is  the {\it Cantor part} of $Du$. {If $u\in BV(\Omega)$, then $\nabla u(x)$ is the  \emph{approximate gradient} of $u$ for a.e. $x\in\Omega$:
\[
\lim_{\rho\to 0}\-int_{B_\rho(x)}\frac{|u(y)-u(x)-\nabla u(x)(y-x)}{|y-x|}\,\mathrm{d}y=0\,.
\]}
For the sake of simplicity, we denote by $D^su =D^c u + (u^+-u^-)\otimes \nu_u\,\mathcal{H}^{d-1}\ristretto{J_u}$.
{{If $u\in{BV}(\Omega)$, then $\mathcal{H}^{d-1}(S_u\setminus J_u)=0$; so in the sequel we shall essentially identify the two sets.}}

We recall that the space ${SBV}(\Omega)$ of {\it special functions of bounded variation} is defined as the set
of all $u\in BV(\Omega)$ such that $D^su$ is concentrated on $S_u$; i.e., $|D^su|(\Omega\setminus S_u)=0$. Finally, for $p>1$ the space $SBV^p(\Omega)$ is the set of $u\in SBV(\Omega)$ with $\nabla u\in L^p(\Omega;\R^d)$ and $\mathcal{H}^{d-1}(S_u)<\infty$.

\subsection{The space $SBV^{p(\cdot)}$. Poincar\'e-type inequality} \label{sec:gsbvp(x)}

We denote by $SBV^{p(\cdot)}(\Omega)$ the set of functions $u\in SBV(\Omega)$ with $\nabla u\in L^{p(\cdot)}(\Omega;\mathbb{R}^{d})$ and $\mathcal{H}^{d-1}(S_u)<+\infty$. 



In order to state a Poincar\'e-Wirtinger inequality in $SBV^{p(\cdot)}$, we first fix some notation, following \cite{CL}. With given $a$, $b\in\R$, we denote $a\wedge b:=\min(a,b)$ and $a\vee b:=\max(a,b)$. Let $B$ be a ball in $\R^d$. For every measurable function $u:B\to\R$, we set
\begin{equation*}
u_*(s;B):= \inf\{t\in\R:\,\, \mathcal{L}^d(\{u<t\}\cap B)\geq s\} \qquad \mbox{ for } 0\leq s \leq \mathcal{L}^d(B),
\end{equation*}
and
\begin{equation*}
\quad {\rm med}(u;B):=u_*\left(\frac{1}{2}\mathcal{L}^d(B);B\right).
\end{equation*}

For every $u\in SBV^{p(\cdot)}(\Omega)$ such that
\begin{equation*}
\left(2\gamma_{\rm iso}\mathcal{H}^{d-1}(S_u\cap B)\right)^{\frac{d}{d-1}} \leq \frac{1}{2}\mathcal{L}^d(B)\,,
\end{equation*}
we define
\begin{equation*}
\begin{split}
\tau'(u;B) & := u_*\left(\left(2\gamma_{\rm iso}\mathcal{H}^{d-1}(S_u\cap B)\right)^{\frac{d}{d-1}};B\right)\,, \\
\tau''(u;B) & := u_*\left(\mathcal{L}^d(B)-\left(2\gamma_{\rm iso}\mathcal{H}^{d-1}(S_u\cap B)\right)^{\frac{d}{d-1}};B\right)\,, 
\end{split}
\end{equation*}
and the truncation operator
\begin{equation}
T_Bu(x):= (u(x)\wedge \tau''(u;B)) \vee \tau'(u;B)\,,
\label{eq:truncated}
\end{equation}
where $\gamma_{\rm iso}$ is the dimensional constant in the relative isoperimetric inequality. 

For any $M>0$, we also define  
\begin{equation}
{u}^M:= M \wedge { u} \vee (-M)\,.
\label{eq:classictruncation}
\end{equation}

We recall the following Poincar\'e-Wirtinger inequality for $SBV$ functions with small jump set in a ball, which was first proven in the scalar setting in \cite[Theorem~3.1]{DeGCarLea}, and then extended to vector-valued functions in \cite[Theorem~2.5]{CL}.

\begin{theorem}
Let $u\in SBV(B)$ and assume that
\begin{equation}
\left(2\gamma_{\rm iso}\mathcal{H}^{d-1}(S_u\cap B)\right)^{\frac{d}{d-1}} \leq \frac{1}{2}\mathcal{L}^d(B)\,.
\label{(10)}
\end{equation}
If $1\leq p < d$ then the function $T_Bu$ satisfies $|DT_Bu(B)|\le 2\int_B|\nabla u|\,\mathrm{d}y$, 
\begin{equation}
\left(\int_B |T_Bu-{\rm med}(u;B)|^{p^*}\,\mathrm{d}x\right)^\frac{1}{p^*} \leq \frac{2\gamma_{\rm iso}p(d-1)}{d-p} \left(\int_B|\nabla u|^p\,\mathrm{d}x\right)^\frac{1}{p},
\label{(11)}
\end{equation}
and
\begin{equation}
\mathcal{L}^d(\{T_Bu\neq u\}\cap B) \leq 2 \left(2\gamma_{\rm iso}\mathcal{H}^{d-1}(S_u\cap B)\right)^{\frac{d}{d-1}}\,,
\label{(12)}
\end{equation}
where $p^*:=\frac{dp}{d-p}$.
If $p\ge d$, then, for any $q\ge 1$,
\begin{equation}\label{maggd}
{\left(\int_B |T_Bu-{\rm med}(u;B)|^{q}\,\mathrm{d}x\right)^\frac{1}{q} \leq c(q,N,\gamma_{\rm iso})(\mathcal{L}^d(B))^{\frac{1}{q}+\frac{1}{N}-\frac{1}{p}} \left(\int_B|\nabla u|^p\,\mathrm{d}x\right)^\frac{1}{p}\,.}
\end{equation}
\label{thm:poincsbv}
\end{theorem}

{As a first application of Theorem~\ref{thm:poincsbv} one can obtain the following sufficient condition for the existence of the approximate limit at a given point (see \cite[Theorem 7.8]{Ambrosio-Fusco-Pallara:2000}).
\begin{theorem}
Let $u\in SBV_{\rm loc}(\Omega)$ and $x\in\Omega$. If there exist $p,q>1$ such that
\begin{equation*}
\lim_{\rho\to0} \frac{1}{\rho^{d-1}} \left[\int_{B_\rho(x)}|\nabla u|^p\,\mathrm{d}y + \mathcal{H}^{d-1}(S_u\cap B_\rho(x)) \right] =0 \quad \mbox{and} \quad \mathop{\lim\sup}_{\rho\to0} \dashint_{B_\rho(x)} |u(y)|^q \,\mathrm{d}y <\infty\,,
\end{equation*}
then $x\not\in S_u$.
\label{thm:thm7.8AFP}
\end{theorem}}

{Another consequence of Theorem~\ref{thm:poincsbv} is the following compactness result, which is a slight extension of the result established in \cite[Theorem~2.8]{SSS} (see also \cite[Theorem~4.1]{DCLV} for a related result under the additional stronger assumption \eqref{eq:loghold}).} 
Motivated by the blow-up analysis of Lemma~\ref{lem:decay}, we will prove the result for a fixed ball and a uniformly convergent sequence of continuous variable exponents $p_h:B\to (1,+\infty)$ satisfying:
\begin{equation}\label{eq:assump}
1<p^-\le p_h(y)\le p^+<+\infty,\ \ \ \ \ \ \forall y\in B,
\end{equation}
where $p^-$ and $p^+$ satisfy
\begin{equation}\label{eq:assump2}
p^-<d \ \hbox{ and }\  p^+< (p^-)^* \ \hbox{ or } \ p^-\ge d.
\end{equation}

\begin{theorem}\label{thm:3.5}
Let $B\subset\Omega$ be a ball, $(p_h)_{h\in\N}$ be a sequence of variable exponents $p_h:B\to(1,+\infty)$ complying with \eqref{eq:assump} and \eqref{eq:assump2} and converging uniformly to some $\bar{p}:B\to(1,+\infty)$ in $B$. Let $\{u_h\}_{h\in\mathbb{N}}\subset SBV^{p_h(\cdot)}(B)$ be such that
\begin{equation}
\sup_{h\in\mathbb{N}}\int_B |\nabla u_h|^{p_h(y)}\,\mathrm{d}y < +\infty\,,\quad \lim_{h\to+\infty} \mathcal{H}^{d-1}(S_{u_h}\cap B)=0\,.
\label{eq:ebounded}
\end{equation}
Then there exist a function $u_0\in W^{1,\bar{p}(\cdot)}(B)$ and a subsequence (not relabeled) of $\{u_h\}$ such that
\begin{equation}\label{eq:claims}
\begin{split}
& (i) \,\,\lim_{h\to+\infty}\int_B |T_B{u}_h-{\rm med}(u_h;B)-u_0|^{p_h(y)}\,\mathrm{d}y=0\,,\\
 &(ii) \, {u}_h-{\rm med}(u_h;B) \to u_0 \quad  \mathcal{L}^d-\hbox{a.e. in $B$}\,,\\
&(iii)\int_B|\nabla u_0|^{\bar {p}(y)}\,\mathrm{d}y\le \liminf_{h\to+\infty}\int_B|\nabla u_h|^{p_h(y)}\,\mathrm{d}y\,,\\
&(iv) \lim_{h\to+\infty}{\mathcal L}^d(\{T_B u_h\neq u_h\}\cap B)=0\,.
\end{split}
\end{equation}
\end{theorem}

\begin{proof}
The extension with respect to Theorem 2.8 in \cite{SSS} regards only the lower semicontinuity inequality in \eqref{eq:claims}. We repeat the argument, since this gives us the occasion to develop some details.

\noindent First, observe that ${\bar p}(\cdot)$ also satisfies \eqref{eq:assump} with the same constants. We set for brevity $\bar{u}_h:=T_B{u}_h-{\rm med}(u_h;B)$ and we distinguish two cases, according to the values of $p^-$ and $p^+$. 

{\bf Step~1.} Here we assume that
\[
p^-<d \hbox{ and }p^+< (p^-)^*.
\]
By Theorem \ref{thm:poincsbv} and \eqref{eq:ebounded} we have
\begin{equation}
\label{boundpoin}
\sup_{h\in\mathbb{N}} \|\bar{u}_h\|_{L^{(p^-)^*}(B)}  <+\infty
\end{equation}
and $|D \bar{u}_h|(B)\le 2\int_B|\nabla u_h|\,\mathrm{d}y$. By the compactness theorem for $BV$ functions (see for example \cite[Theorem 3.23]{Ambrosio-Fusco-Pallara:2000})  there exists a function $u_0\in BV(B)$ and a subsequence (not relabeled) of ${\bar u}_h$ such that $\bar{u}_h \rightharpoonup u_0$ weakly$^{*}$ in $BV$ (thus, in particular, in $L^1$). 
Now, for every $M>0$, let ${\bar u}_h^M$ and $u_0^M$ be the truncations of ${\bar u}_h$ and $u_0$, respectively, according to \eqref{eq:classictruncation}. 
Then, applying the compactness theorem in $SBV$ (\cite[Theorem 4.8]{Ambrosio-Fusco-Pallara:2000}) to ${\bar u}_h^M$ we get that $u_0^M\in SBV(B)$, ${\bar u}_h^M\to u_0^M$ in $L^q$, for any $q>1$, $\nabla {\bar u}_h^M\rightharpoonup\nabla u_0^M$ weakly in $L^1$ and 
\begin{equation}\label{nojump}
\mathcal{H}^{d-1}(S_{u_0^M}) \leq \mathop{\lim\inf}_{h\to+\infty} \mathcal{H}^{d-1}(S_{{\bar u}_h^M})=0\,,
\end{equation}
meaning that $u_0^M\in W^{1,1}(B)$. Now we apply De Giorgi's semicontinuity Theorem (\cite{DeG}) to the integral functional $\int_Bf(p(y), w(y))\,\mathrm{d}y$ defined by $f:\mathbb{R}\times\mathbb{R}^d\to [0,+\infty)$ as 
$f(p,\xi)=|\xi|^{p\vee 1}$.
The continuous function $f$ is convex in $\xi$, thus ensuring the sequential lower semicontinuity of the functional  whenever $p_j$ strongly converges in $L^1(B)$ and $w_j$ weakly converges in $L^1(B,\mathbb{R}^d)$. 
This implies that 
\[
\begin{split}
\int_B|\nabla u_0^M|^{{\bar p}(y)}\,\mathrm{d}y&\le \liminf_{h\to+\infty}\int_B |\nabla {\bar u}_h^M|^{p_h(y)}\,\mathrm{d}y\le \liminf_{h\to+\infty}\int_B|\nabla {\bar u}_h|^{p_h(y)}\,\mathrm{d}y,
\end{split}
\]
and $\nabla u_0^M$ is equibounded in $L^{{\bar p}(\cdot)}(B,\mathbb{R}^d)$. Therefore, letting $M\to+\infty$, 
we obtain that $u_0\in W^{1,{\bar p}(\cdot)}(B)$ and {\eqref{eq:claims}$(iii)$ follows.} 
Thanks to \eqref{boundpoin},  since $p^+<(p^-)^*$, we get that   $|\bar{u}_h|^{p^+}$ is equiintegrable, hence $\bar{u}_h$ strongly converges to $u_0$ in $L^{p^+}(B)$. {Now, \eqref{eq:claims}$(i)$ follows from Lemma \ref{embedding}, while \eqref{eq:claims}$(ii)$ and $(iv)$ can be inferred from \eqref{(12)} and \eqref{eq:ebounded}.} 

{\bf Step~2.} Here we assume $p^-\ge d$. In this case, the proof goes exactly as in {\bf Step~1} using \eqref{maggd} of Theorem  \ref{thm:poincsbv} instead of \eqref{(11)}.

\end{proof}

\begin{remark}\label{pcost}
{We observe that the previous assumptions \eqref{eq:assump} and \eqref{eq:assump2} are always satisfied} if the sequence $p_h(\cdot)$ converges uniformly to a constant function $\bar p$. 
In fact, if ${\bar p}\le d$, then we can find $\eta>0$ such that, for $h$ large, 
\[
{\bar p}-\eta<p_h(\cdot)<{\bar p}+\eta\ \hbox{ and }\ {\bar p}+\eta<({\bar p}-\eta)^*.
\]
If ${\bar p}>d$, then we can find $\eta>0$ such that, for $h$ large, 
\[
d<{\bar p}-\eta<p_h(\cdot)<{\bar p}+\eta,
\]
and these are the two cases  in the previous theorem.
\end{remark}

\subsection{Lusin approximation in $SBV^{p(\cdot)}$} \label{sec:lusinsbpx}

Let $\mu$ be a positive, finite Radon measure in $\R^d$. The \emph{maximal function} is defined as
\begin{equation*}
M(\mu)(x):= \sup_{\varrho>0} \frac{\mu(B_\varrho(x))}{\mathcal{L}^d(B_\varrho)}\,,\quad x\in \R^d\,.
\end{equation*}
As a consequence of the Besicovitch covering theorem (see, e.g., \cite{EG}), it can be shown that
\begin{equation}
\mathcal{L}^d(\{x\in\R^d:\,\, M(\mu)(x)>\lambda\}) \leq \frac{c}{\lambda} \mu(\R^d)
\label{eq:2.3dclv}
\end{equation}
for a constant $c$ depending only on $d$. 

We recall the Lipschitz truncation result for $SBV^{p(\cdot)}$ functions, proved in \cite[Theorem~3.1]{DCLV} (see also \cite{DMS}), which we state here in a slightly modified version suitable for our purposes.
 
\begin{theorem}\label{thm:dclvLip}
Let $\Omega\subset\R^d$ be an open bounded set with Lipschitz boundary, and let $p:\Omega\to(1,+\infty)$ be log-H\"older continuous on $\Omega$ and $1<p^-\le p(x)\le p^+<\infty$, for every $x\in\Omega$. Let $\{v_h\}_{h\in\N}\subset SBV^{p(\cdot)}(\Omega)$ be a sequence of functions with compact support in $\Omega$, such that
\begin{equation*} 
\|v_h\|_1\to0\,,\quad \sup_{h\in\N}\|v_h\|_\infty <+\infty\,, \quad \sup_{h\in\N}\int_\Omega |\nabla v_h|^{p(x)}\,\mathrm{d}x <+\infty \,.
\end{equation*}
Let $\{\theta_h\}_{h\in\N}$ be a sequence of strictly positive numbers such that $\theta_h\to0$ and $\frac{\|v_h\|_{L^{p(\cdot)}(\Omega)}}{\theta_h}\to0$. Then for every $h$ there exist sequences $\mu_j, \lambda_{h,j}>1$ such that for every $h,j\in\N$ 
\begin{equation*}
2^j\le\mu_j \leq \lambda_{h,j} \leq \mu_{j+1}\,,
\end{equation*}
and there exists a sequence $\{v_{h,j}\} \subset W^{1,\infty}(\Omega)$ such that for every $h,j\in\N$
\begin{equation*}
\|v_{h,j}\|_\infty \leq \theta_h\,,\quad \|\nabla v_{h,j}\|_\infty \leq C \lambda_{h,j}
\end{equation*}
for some constant $C$ depending on $d,p^-,p^+$, the $\log$-H\"older constant of $p(\cdot)$, and also on $v_h$ in terms of $\|v_h\|_{W^{1,p(\cdot)}(\Omega)}$. Moreover, up to a null set, 
\begin{equation*}
\{v_{h,j}\neq v_h\} \subset \{M(|v_h|)>\theta_h\} \cup \{M(|Dv_h|)>3K\lambda_{h,j}\} 
\end{equation*} 
for some constant $K$ depending on $d,p^-,p^+$, the $\log$-H\"older constant of $p(\cdot)$, and also on $v_h$ in terms of $\|v_h\|_{W^{1,p(\cdot)}(\Omega)}$. For every $j\in \N$
\begin{equation*}
\nabla v_{h,j} \rightharpoonup^* 0 \quad \mbox{ in $L^\infty(\Omega;\R^{d})$ as $h\to+\infty$. }
\end{equation*}
Finally, there exists a sequence $\varepsilon_j>0$ with $\varepsilon_j\to0$ such that for every $h,j\in\N$,
\begin{equation*}
\begin{split}
& \|\nabla v_{h,j}\chi_{\{M(|v_h|)>\theta_h\} \cup \{M(|\nabla v_h|)>2K\lambda_{h,j}\} } \|_{L^{p(\cdot)}(\Omega)} \\
& \leq C \|\lambda_{h,j}\chi_{\{M(|v_h|)>\theta_h\} \cup \{M(|\nabla v_h|)>2K\lambda_{h,j}\} } \|_{L^{p(\cdot)}(\Omega)} \leq C \frac{\|v_h\|_{L^{p(\cdot)}(\Omega)}}{\theta_h} \mu_{j+1} + \varepsilon_j \,.
\end{split}
\end{equation*}
\end{theorem}



\subsection{Free-discontinuity functionals with $p(\cdot)$-growth}\label{sec:freepx}
In this paragraph we consider integral functionals of the form 
\begin{equation}\label{funct}
F(u,c,A) :=\int_Af(x,\nabla u)+c{\mathcal H}^{d-1}(S_u\cap A),
\end{equation}
defined on $SBV_{\rm loc}(\Omega)$, {where $c>0$ and $A\subset\Omega$ is an open set.} The Carath\'eodory function $f:\Omega\times\mathbb{R}^d\to\mathbb{R}$ will be supposed to satisfy the following growth condition:
\begin{equation}\label{cresci}
L^{-1}|\xi|^{p(x)}\le f(x,\xi)\le L(1+|\xi|^{p(x)}),
\end{equation}
for any $\xi\in\mathbb{R}^d$, a.e. $x\in\Omega$, {where $L \ge 1$ and the variable exponent $p:\Omega\to(1,+\infty)$ is 
a bounded function. } We will write $F(u,A)$ for $F(u,1,A)$.

{We recall the classical definition of \emph{deviation from minimality} (see, e.g., \cite{Ambrosio-Fusco-Pallara:2000}), which gives an estimate of how far is $u$ from being a minimizer of $F$ in $\Omega$.}

\begin{definition}[Deviation from minimality]\label{deviazione}
The deviation from minimality  of a function $u\in SBV_{\rm loc}(\Omega)$ satisfying $F(u,c,A)<+\infty$ for all open sets $A\ssubset\Omega$ is defined as the smallest $\lambda\in [0,+\infty]$ such that
\[
F(u,c,A)\le F(v,c,A) + \lambda
\]
for all $v\in SBV_{\rm loc}(\Omega)$ satisfying $\{v\neq u\} \ssubset A \ssubset \Omega$. It is denoted by ${\rm Dev}(u,c,\Omega)$.

\end{definition}

If ${\rm Dev}(u,c,\Omega)=0$ we say that $u$ is a \emph{local minimizer} of $F(u,c,\Omega)$ in $\Omega$.

%

The following lemma compares the energy of a function $u$ in a ball $B_{\rho'}$ with the energy of the function $v\chi_{B_\rho} + u\chi_{B_{\rho'}\setminus B_\rho}$ , where $\rho<\rho'$. The proof is a straightforward adaptation of \cite[Lemma 7.3]{Ambrosio-Fusco-Pallara:2000}.

\begin{lemma}\label{lemma:7.3}
Let $\Omega \subset \R^d$ be open and bounded and let  $B_R\ssubset\Omega$, $\rho<\rho'<R$.  For every functional $F$  of the type \eqref{funct} with $f$  satisfying \eqref{cresci}, and for every $u, v\in SBV_{\rm loc}(B_R)$ such that $F(u,c,B_{\rho'})<+\infty$, $F(v,c,B_{\rho'})<+\infty$, and ${\mathcal H}^{d-1}(S_v\cap \partial B_\rho)=0$, then
\begin{equation}\label{eq: 7.3}
\begin{split}
&F ( u,c, B_\rho) \le  F(v,c,B_\rho)  + c\mathcal{H}^{d-1}(\{{\tilde v}\neq{\tilde u}\}\cap\partial B_\rho) +{\rm Dev}(u,c,B_{\rho'})\,,\\
&{\rm Dev}(v,c, B_\rho)\le F(v,c,B_\rho)-F(u,c,B\rho)+c\mathcal{H}^{d-1}(\{{\tilde v}\neq{\tilde u}\}\cap\partial B_\rho) +{\rm Dev}(u,c,B_{\rho'})\,.
\end{split}
\end{equation}
\end{lemma}

\section{Asymptotic behavior of almost minimizers with small jump sets} \label{sec:smalljump}
This section is devoted to the proof of an auxiliary $\Gamma$-convergence type result for sequences of functionals with variable growth. We will namely consider a sequence of continuous functions $f_h:B_R\times\mathbb{R}^d\to [0,+\infty)$ satisfying
\begin{equation}\label{crescih}
L^{-1}|\xi|^{p_h(x)}\le f_h(x,\xi)\le L(1+|\xi|^{p_h(x)})\,,
\end{equation}
for any $\xi\in\mathbb{R}^d$, $x\in\Omega$, $p_h:\Omega\to(1,+\infty)$ {uniformly log-H\"older continuous satisfying \eqref{eq:loghold}} 
and \eqref{eq:assump}, and $L\ge 1$.
If $v\in SBV_{\rm loc}(\Omega)$, $c>0$, we set 
\[
F_h(v,c, B_\rho):=\int_{B_\rho}f_h(x,\nabla v)\,\mathrm{d}x+c{\mathcal H}^{d-1}(S_v\cap B_\rho)\,,
\]
and we write ${\rm Dev}_h(v, c, B_\rho)$ for the corresponding deviation from minimality. 

The next theorem describes the behavior of a sequence $u_h$ of ``almost minimizers'' of functionals $F_h$ when  the functions $f_h$ converge  to $f$ uniformly on compact sets 
and the measures of the discontinuity sets $S_{u_h}$ are infinitesimal. The proof follows the lines of the proof of Theorem 2.6 in \cite{FMT}, taking advantage of a lower semicontinuity result proved in \cite[Theorem 4.1]{DCLV} in the $SBV$ setting of variable growth, here slightly modified.

\begin{theorem}\label{gammaconv}
Let $f_h:B_R\times\mathbb{R}^d\to [0,+\infty)$ be a sequence of continuous functions, convex with respect to the second variable, satisfying \eqref{crescih}, with $p_h$ converging uniformly to some ${\bar p}\in (1,+\infty)$ in $B_R$. Let $u_h\in SBV(B_R)$, $m_h:={\rm med}(u_h,B_R)$, $c_h\in (0,+\infty)$. Assume that
\begin{equation}\label{hp:1}
\sup_h F_h(u_h,c_h,B_R)<+\infty\,,
\end{equation}
\begin{equation}\label{hp:2}
\lim_{h\to+\infty}{\mathcal H}^{d-1}(S_{u_h})=0\,,
\end{equation}
\begin{equation}\label{hp:3}
\lim_{h\to+\infty}{\rm Dev}_h(u_h,c_h,B_R)=0\,,
\end{equation}
\begin{equation}\label{hp:4}
\lim_{h\to+\infty}u_h-m_h=u_0\in W^{1,{\bar p}}(B_R) \  \hbox{ $\mathcal{L}^d$-a.e. in }B_R. 
\end{equation}
Then, if $f_h$ converges to $f:\mathbb{R}^d\to [0,+\infty)$, with $L^{-1}|\xi|^{\bar p}\le f(\xi)\le L|\xi|^{\bar p}$, uniformly on compact subsets of $B_R\times\mathbb{R}^d$, the function $u_0$ is a local minimizer of the functional $v\to\int_{B_R}f(\nabla v)\,\mathrm{d}x$ in $W^{1,{\bar p}}(B_R)$ and
\begin{equation}
\lim_{h\to+\infty}F_h(u_h,c_h, B_\rho)=\int_{B_\rho}f(\nabla u_0)\,\mathrm{d}x\,, \ \ \ \forall \rho\in(0,R)\,.
\end{equation}
\end{theorem}
\begin{proof}
Replacing  $u_h$ with $u_h-m_h$, we may always assume that $m_h =0$ for all $h$. {By virtue of assumptions \eqref{hp:1} and \eqref{hp:2} we may appeal to Theorem \ref{thm:3.5}  (see also Remark \ref{pcost}). }
Thanks to this and to an inspection of its proof, we have that $u_0\in W^{1,{\bar p}}(B_R)$, 
 $\nabla{\bar u}_h^M\rightharpoonup\nabla u_0^M$ in $L^1$, 
 \begin{equation}\label{convubar}
 {\mathcal L}^d(B_R\cap\{u_h\neq{\bar u}_h\})\to 0\,,\ \ \hbox{ and }\  \lim_{h\to+\infty}\int_B |{\bar u}_h-u_0|^{p_h(y)}\,\mathrm{d}y=0\,,
 \end{equation}
where we still used the notation ${\bar u}_h:=T_{B_R}u_h$.
We divide the proof into three steps.

{\bf Step 1: a lower semicontinuity result.} We claim that for a.e. $\rho\in (0,R)$, it holds
\begin{equation}\label{lsc1}
\int_{B_\rho} f(\nabla u_0)\,\mathrm{d}y\le \liminf_{h\to+\infty}\int_{B_\rho}f_h(y,\nabla {\bar u}_h)\,\mathrm{d}y\,,
\end{equation}
which in turn, taking into account \eqref{hp:2}, implies 
\begin{equation}\label{semiFh}
\int_{B_\rho} f(\nabla u_0)\,\mathrm{d}y\le \liminf_{h\to+\infty} F_h({\bar u}_h,c_h,B_\rho)\,.
\end{equation} 
If we show the lower semicontinuity inequality \eqref{lsc1} for $\nabla {\bar u}_h^M \rightharpoonup \nabla u_0^M$ weakly in $L^1$, the claim will be proved, since $f_h(y,0)\to f(0)=0$ uniformly on $B_\rho$. So, replacing ${\bar u}_h$ with  ${\bar u}_h^M$ we can suppose that ${\bar u}_h$ is bounded uniformly with respect to $h$. 
Let us denote $v_h:={\bar u}_h-u_0$, so that $v_h\to 0$ in $L^q(B_R)$, for every $q\ge 1$.  Now we adapt the proof of \cite[Theorem 4.1]{DCLV} to our setting. Even if only a few changes are significant, we will write it for the sake of completeness. 

Possibly multiplying $v_h$ by a cut-off function $\zeta\in C^\infty(\mathbb{R}^d)$, with compact support in $B_R$ and identically equal to $1$ in $B_\rho$, we can assume, without loss of generality, that the functions $v_h\in SBV(\mathbb{R}^d)$ have compact support in $B_R$. Fix $\eta>0$ such that $p_\eta:={\bar p}-\eta>1$ and $\|p_h-{\bar p}\|_{L^\infty(B_R)}<\eta$, for $h$ large enough, so that 
\[
\sup_h\int_{B_R}|\nabla v_h|^{p_\eta}\,\mathrm{d}y\le \sup_h\left[\int_{B_R}|\nabla v_h|^{p(x)}+\mathcal{L}^{d}(B_R)|\right]<+\infty\,, 
\]
\[
\ \ \ \gamma_h :=\|v_h\|_{L^{p_\eta}(B_R)}\to 0\,,\ \ \ \  \sup_h\|v_h\|_{L^\infty(B_R)}<+\infty\,.
\]
So we can apply 
Theorem~\ref{thm:dclvLip} to $v_h$: given $\theta_h$, with $\gamma_h/\theta_h\to 0$, there exist  sequences $\varepsilon_j\in (0,1), \mu_j,\lambda_{h,j}>1$, with $\varepsilon_j\to 0$, $\mu_j\to+\infty$,  and $\mu_j\le \lambda_{h,j}\le \mu_{j+1}$,  and there exists a sequence $v_{h,j}\in W^{1,\infty}(\mathbb{R}^d)$ such that $v_{h,j}=0$ in $\mathbb{R}^d\setminus B_R$, $\|v_{h,j}\|_\infty\le \theta_h$, $\|\nabla v_{h,j}\|_\infty\le C\lambda_{h,j}$, $\{v_{h,j}\neq v_h\}\subset B_R\cap\{|M(|v_h|)>\theta_h\}\cup\{M(|Dv_h|)>3K\lambda_{h,j}\}$, and 
\begin{equation}\label{misure}
\lambda_{h,j}^{p_\eta}{\mathcal L}^d(B_R\cap\{|M(|v_h|)>\theta_h\}\cup\{M(|\nabla v_h|)>2K\lambda_{h,j}\})\le C\left[\frac{\gamma_h}{\theta_h}\mu_{j+1}+\varepsilon_j\right]^{p_\eta},
\end{equation}
with $K$ and $C$ two constants independent of $h,j,\eta$. Moreover, for every $j\in\mathbb{N}$ and for $h\to\infty$,
$$
\nabla v_{h,j}\rightharpoonup 0 \ \ \ \hbox{ weakly$^*$ in } L^\infty(B_R,\mathbb{R}^d).
$$
For any $r>0$, we have 
\[
\int_{B_\rho} f_h(y,\nabla {\bar u}_h)\,\mathrm{d}y=\int_{B_\rho}f_h(y,\nabla u_0+\nabla v_h)\,\mathrm{d}y\ge\int_{E_r\cap\{v_h=v_{h,j}\}}f_h(y,\nabla u_0+\nabla v_{h,j})\,\mathrm{d}y,
\]
where $E_r=\{y\in B_\rho : |\nabla u_0|\le r\}$. Recalling the convergence of $f_h$ to $f$, we deduce that
\[
\begin{split}
\int_{B_\rho}& f_h(y,\nabla {\bar u}_h)\,\mathrm{d}y\ge \int_{E_r\cap\{v_h=v_{h,j}\}}\left[f_h(y,\nabla u_0+\nabla v_{h,j})-f(\nabla u_0+\nabla v_{h,j})\right]\,\mathrm{d}y\\
&+\int_{E_r\cap\{v_h=v_{h,j}\}}f(\nabla u_0+\nabla v_{h,j})\,\mathrm{d}y=\int_{E_r\cap\{v_h=v_{h,j}\}}\left[f_h(y,\nabla u_0+\nabla v_{h,j})-f(\nabla u_0+\nabla v_{h,j})\right]\,\mathrm{d}y\\
&+\int_{E_r}f(\nabla u_0+\nabla v_{h,j})\,\mathrm{d}y-\int_{E_r\cap\{v_h\neq v_{h,j}\}}f(\nabla u_0+\nabla v_{h,j})\,\mathrm{d}y\,,
\end{split}
\]
obtaining, when $h\to+\infty$,
\[
\begin{split}
\liminf_{h\to+\infty}\int_{B_\rho} f_h(y,\nabla {\bar u}_h)\,\mathrm{d}y&\ge \liminf_{h\to+\infty}\int_{E_r}f(\nabla u_0+\nabla v_{h,j})\,\mathrm{d}y\\
&-\limsup_{h\to+\infty}\int_{E_r\cap\{v_h\neq v_{h,j}\}}f(\nabla u_0+\nabla v_{h,j})\,\mathrm{d}y\\
&\ge \int_{E_r}f(\nabla u_0)\,\mathrm{d}y-L\limsup_{h\to+\infty}\int_{E_r\cap\{v_h\neq v_{h,j}\}}|\nabla u_0+\nabla v_{h,j}|^{\bar p}\,\mathrm{d}y\,,
\end{split}
\]
where the De Giorgi's semicontinuity Theorem (\cite{DeG}) and the growth assumption on $f$ have been applied for the last inequality. We now treat the last term:
\[
\begin{split}
&\int_{E_r\cap\{v_h\neq v_{h,j}\}}|\nabla u_0+\nabla v_{h,j}|^{\bar p}\,\mathrm{d}y\le c\left(r^{\bar p}+\lambda_{h,j}^{\bar p}\right){\mathcal L}^d(\{y\in B_R: v_h\neq v_{h,j}\})\\
&\le c\left(r^{\bar p}+\lambda_{h,j}^{\bar p}\right)\left[{\mathcal L}^d(B_R\cap\{M(|v_h|>\theta_h\}\cup\{M(|\nabla v_h|)>2K\lambda_{h,j}\})\right.\\
&\qquad\qquad\qquad\qquad\left.+{\mathcal L}^d(B_R\cap \{M(|D^sv_h|)>K\lambda_{h,j}\})\right].
\end{split}
\]
{Recalling \eqref{eq:2.3dclv}, } 
we have 
\[
c\left(r^{\bar p}+\lambda_{h,j}^{\bar p}\right){\mathcal L}^d(B_R\cap \{M(|D^sv_h|)>K\lambda_{h,j}\})\le c\left(r^{\bar p}+\lambda_{h,j}^{\bar p}\right)\frac{1}{\lambda_{h,j}}{\mathcal H}^{d-1}(S_{v_h}),
\]
which is infinitesimal as $h\to+\infty$, since $\mu_j\le \lambda_{h,j}\le \mu_{j+1}$. Thus we are left to estimate
\[
{\mathcal I}_{h,j} :=c\left(r^{\bar p}+\lambda_{h,j}^{\bar p}\right){\mathcal L}^d(B_R\cap\{M(|v_h|>\theta_h\}\cup\{M(|\nabla v_h|)>2K\lambda_{h,j}\}),
\]
for which we can use \eqref{misure}, so that
\[
\begin{split}
{\mathcal I}_{h,j}\le c\frac{r^{\bar p}+\lambda_{h,j}^{\bar p}}{\lambda_{h,j}^{p_\eta}}\left[\frac{\gamma_h}{\theta_h}\mu_{j+1}+\varepsilon_j\right]^{p_\eta}\le c\frac{r^{\bar p}}{\mu_j^{p_\eta}}\left[\frac{\gamma_h}{\theta_h}\mu_{j+1}+\varepsilon_j\right]^{p_\eta}+c\mu_{j+1}^{\eta}\left[\frac{\gamma_h}{\theta_h}\mu_{j+1}+\varepsilon_j\right]^{p_\eta},
\end{split}
\]
whence
\[
\limsup_{j\to+\infty}\lim_{\eta\to 0}\limsup_{h\to+\infty}{\mathcal I}_{h,j}=0.
\]
In conclusion, we obtained
\[
\liminf_{h\to+\infty}\int_{B\rho} f_h(y,\nabla {\bar u}_h)\,\mathrm{d}y\ge  \int_{E_r}f(\nabla u_0)\,\mathrm{d}y,
\]
and letting $r$ tend to $+\infty$, we proved the lower semicontinuity result \eqref{lsc1}.

{\bf Step 2: asymptotics.} Now we integrate  ${\mathcal H}^{d-1}(\{{\tilde{\bar u}}_h\neq\tilde{u}_h\}\cap\partial B_\rho)$ with respect to $\rho$, and using coarea formula and \eqref{(12)}, we obtain
\[
\begin{split}
a_h:=c_h\int_0^R{\mathcal H}^{d-1}(\{{\tilde{\bar u}}_h\neq\tilde{u}_h\}\cap\partial B_\rho)\,\mathrm{d}\rho&=c_h{\mathcal L}^d(B_R\cap \{{\bar u}_h\neq u_h\})\\
&\le  2 c_h\left(2\gamma_{\rm iso}\mathcal{H}^{d-1}(S_{u_h}\cap B_R)\right)^{\frac{d}{d-1}}.
\end{split}
\]
 We will prove that $\displaystyle{\lim_{h\to+\infty}  a_h=0}$.
{We need to distinguish two cases, according to the value of the limit $c_\infty:=\lim_hc_h$, which exists up to a subsequence (not relabeled). If $c_\infty<+\infty$, the assertion immediately follows from \eqref{hp:2}. If $c_\infty=+\infty$, from \eqref{hp:1} we have
\[
c_h\left(2\gamma_{\rm iso}\mathcal{H}^{d-1}(S_{u_h}\cap B_R)\right)^{\frac{d}{d-1}}\le c_h\left(2\gamma_{\rm iso}\frac{M}{c_h}\right)^{\frac{d}{d-1}}\le \left(\frac{1}{c_h}\right)^{\frac{1}{d-1}}\left(2\gamma_{\rm iso}{M}\right)^{\frac{d}{d-1}}
\]
and again the thesis follows.}

Then, up to a subsequence, we may assume that, for almost every $\rho\in (0,R)$,
\begin{equation}\label{limit}
\lim_{h\to+\infty}c_h\,{\mathcal H}^{d-1}(\{{\tilde{\bar u}}_h\neq\tilde{u}_h\}\cap\partial B_\rho)=0\,.
\end{equation}
Since for any $h$ and for ${\mathcal L}^1$ -a.e. $\rho\in (0,R)$, ${\mathcal H} ^{d-1}(S_{{\bar u}_h}\cap\partial B_\rho)=0$, we can apply Lemma \ref{lemma:7.3}, which gives
\begin{equation}\label{eq:7.3h}
F_h( u_h,c_h, B_\rho) \le  F_h({\bar u}_h,c_h,B_\rho)  + c_h\mathcal{H}^{d-1}(\{{\tilde{\bar u}}_h\neq{\tilde u_h}\}\cap\partial B_\rho) +{\rm Dev}_h(u_h,c_h,B_R)\,,
\end{equation}
and
\begin{equation}\label{eq:7.4h}
\begin{split}
{\rm Dev}_h({\bar u}_h,c_h, B_\rho)&\le F_h({\bar u}_h,c,B_\rho)-F_h(u_h,c_h,B_\rho)+c_h\mathcal{H}^{d-1}(\{{\tilde{\bar u}}_h\neq{\tilde u_h}\}\cap\partial B_\rho)\\
&+{\rm Dev}_h(u_h,c_h,B_{R})\,.
\end{split}
\end{equation}
Moreover, from the growth assumption \eqref{crescih} on $f_h$, taking into account that ${\bar u}_h$ is a truncation of $u_h$, we also have 
\begin{equation}\label{eq:7.5h}
F_h({\bar u}_h,c_h,B_\rho)\le F_h(u_h,c_h,B_\rho)+L\,{\mathcal L}^d(B_\rho\cap\{{\bar u}_h\neq u_h\})\,.
\end{equation}
Thus, if we set for all $\rho<R$,
\[
\alpha(\rho):=\lim_{h\to+\infty}F_h(u_h,c_h,B_\rho)\,,
\]
which exists since the function $\rho\mapsto F_h(u_h,c_h,B_\rho)$ is increasing and equibounded, thanks to \eqref{eq:7.3h}, \eqref{eq:7.5h}, \eqref{limit}, \eqref{hp:3}, and \eqref{convubar}, we may conclude that for ${\mathcal L}^1$-a.e. $\rho\in(0,R)$,
\begin{equation}\label{anche}
\alpha(\rho)=\lim_{h\to+\infty}F_h({\bar u}_h,c_h,B_\rho)\,.
\end{equation}
From this and \eqref{eq:7.4h} we also have that
\[
\lim_{h\to+\infty}{\rm Dev}_h({\bar u}_h,c_h,B_\rho)=0\,.
\]
Now we observe that the sequence of Radon measures 
\[
\mu_h=|\nabla {\bar u}_h|^{p_h(y)}{\mathcal L}^d+c_h{\mathcal H}^{d-1}\lvert{S_{{\bar u}_h}}
\]
is equibounded in mass in view of \eqref{hp:1}, so it weak$^*$ converges (up to a subsequence) to some Radon measure $\mu$ on $B_R$. 

{\bf Step 3: conclusion.} To get the final result, let  $v\in W^{1,\bar p}(B_R)$ be such that $\{v\neq u_0\}\ssubset B_R$. We also consider a regularized function $v^\varepsilon\in W^{1,\infty}(B_R)$ of $v$, strongly converging to $v$ in $W^{1,{\bar p}}(B_R)$. Let $\rho<\rho'\in (0, R)$, with $\rho'$ such that \eqref{anche} holds, $\mu(\partial B_{\rho'})=\mu(\partial B_\rho)=0$ and $\{v\neq u_0\}\ssubset B_{\rho}$. Let $\varphi\in C^\infty_c(B_{\rho'})$ be such that $\varphi = 1$ on $B_\rho$ and define
$\zeta_h= \varphi v^\varepsilon+ (1-\varphi) {\bar u}_h$; since $\{\zeta_h\neq {\bar u}_h\}\ssubset B_{\rho'}$, straightforward computations lead to
\[
\begin{split}
F_h({\bar u}_h,c_h,B_{\rho'})&\le F_h(\zeta_h,c_h,B_{\rho'})+{\rm Dev}_h({\bar u}_h,c_h,B_{\rho'})\\
&\le F_h(v^\varepsilon,c_h,B_\rho)+c\left[\int_{B_{\rho'}\setminus B_\rho}\left(1+|\nabla v^\varepsilon|^{p_h(y)}+|\nabla {\bar u}_h|^{p_h(y)}+\frac{|v^\varepsilon-{\bar u}_h|^{p_h(y)}}{(\rho'-\rho)^{p_h(y)}}\right)\mathrm{d}y\right]\\
&+c_h{\mathcal H}^{d-1}(S_{{\bar u}_h}\cap B_{\rho'}\setminus{\overline B}_\rho) +{\rm Dev}_h({\bar u}_h,c_h,B_{\rho'})\,,
\end{split}
\]
for a suitable constant $c\ge 1$ depending only on $L$ and $p^+, p^-$. Letting $h\to+\infty$ and using \eqref{convubar}, we have
\[
\alpha(\rho')\le \int_{B_\rho}f(\nabla v^\varepsilon)\,\mathrm{d}y+c\left[\int_{B_{\rho'}\setminus B_\rho}\left(1+|\nabla v^\varepsilon|^{\bar p}+\frac{|v^\varepsilon-u_0|^{\bar p}}{(\rho'-\rho)^{\bar p}}\right)\mathrm{d}y\right]+c\,\mu(B_{\rho'}\setminus B_\rho)\,.
\]
Now we  let $\varepsilon\to0$ and recalling that $v=u_0$ outside $B_\rho$, we easily obtain 
\[
\alpha(\rho)\le \alpha(\rho')\le \int_{B_\rho}f(\nabla v)\,\mathrm{d}y+c\left[\int_{B_{\rho'}\setminus B_\rho}\left(1+|\nabla v|^{\bar p}\right)\mathrm{d}y\right]+c\,\mu(B_{\rho'}\setminus B_\rho)\,.
\]
Therefore, letting $\rho'$ tend to $\rho$ we finally get that for ${\mathcal L}^1$ -a.e. $\rho$ and any $v\in W^{1,{\bar p}} (B_R)$ such that $\{v\neq u_0\}\ssubset B_\rho$ we have 
\[
\limsup_{h\to+\infty}F_h({\bar u}_h,c_h,B_\rho)=\limsup_{h\to+\infty}F_h({u}_h,c_h,B_\rho)\le \int_{B_\rho}f(\nabla v)\,\mathrm{d}y\,.
\]
Choosing $v=u_0$ in the previous inequality and taking into account \eqref{semiFh}, we get that
\[
\lim_{h\to+\infty}F_h({u}_h,c_h,B_\rho)= \int_{B_\rho}f(\nabla u_0)\,\mathrm{d}y
\]
and that $u_0$ has the claimed minimizing property. This concludes the proof.

\end{proof} 
\section{Strong minimizers of free-discontinuity functionals with $p(\cdot)$-growth}  \label{sec:strongmin}

\subsection{Assumptions on the energy}\label{sec:mainassump}
Consider a variable exponent $p\colon \Omega \to (1, +\infty)$ satisfying
\begin{equation}\label{eq:assump+}
1<p^-\le p(x)\le p^+<+\infty,\ \ \ \ \ \ \forall x\in B\,.
\end{equation}
For $K\subset\R^d$ be a closed set and $u\in W^{1,p(\cdot)}(\Omega\setminus K)$ we define the functional
\begin{equation}
\mathcal{G}(K,u):= \int_{\Omega\setminus K} f(x,\nabla u)\,\mathrm{d}x + \alpha\int_{\Omega\setminus K}|u-g|^q\,\mathrm{d}x +  \mathcal{H}^{d-1}(K\cap\Omega)
\label{eq:functionalsG}
\end{equation}
where $\alpha>0$, $q\geq1$, and $g\in L^\infty(\Omega)$. Accordingly, we will denote with $\mathcal F$ the weak formulation of $\mathcal{G}$, that is the free-discontinuity functional of the form
\begin{equation}
\mathcal{F}(u):= \int_\Omega f(x,\nabla u)\,\mathrm{d}x + \alpha\int_\Omega|u-g|^q\,\mathrm{d}x +  \mathcal{H}^{d-1}(S_u\cap\Omega)\,, 
\label{eq:functionals}
\end{equation}
defined for $u\in SBV^{p(\cdot)}(\Omega)$.

We assume that the function $f:\Omega\times\mathbb{R}^d\to[0,+\infty)$ has the form 
\begin{equation}\label{effe}
f(x,\xi)=|\xi|^{p(x)}+h(x,\xi),
\end{equation}
where $h$ is a continuous function, convex in $\xi$ and such that
\begin{equation}\label{ipo1h}
0\le  h(x,\xi)\le L(1+|\xi|^{p(x)}),
\end{equation}
for each $(x,\xi)\in\Omega\times\mathbb{R}^d$, and
\begin{equation}\label{ipo2h}
|h(x,\xi)-h(x_0,\xi)|\le L\,\omega(|x-x_0|)(|\xi|^{p(x)}+|\xi|^{p(x_0)})(1+|\log|\xi||)
\end{equation}
for any $\xi\in\mathbb{R}^d$, $x,x_0\in\Omega$ and where $L\ge 1$. Here $\omega:[0,+\infty)\to[0,+\infty)$ is a nondecreasing continuous function such that $\omega(0)=0$, which represents the modulus of continuity of $p(\cdot)$. From now on, we will assume $\omega$ to satisfy \eqref{eq:rinfologhold}.

\subsection{A Decay lemma}\label{sec:decay}

{We start by proving a crucial decay property of the energy $F$ in small balls. The proof is  the variable exponent counterpart of the well-known argument, based on a blow-up procedure, devised for energies with $p$-growth (see, e.g., \cite[Lemma~7.14]{Ambrosio-Fusco-Pallara:2000}). The adaptation to the variable exponent setting is nontrivial, since one of the major ingredients - the homogeneity of the bulk densities - is missing. A first step in this direction, but still for a constant $p$, is provided by \cite[Lemma~3.3]{FMT} with the introduction of a inhomogeneous correction, namely a function $h$ as above. }

Before stating and proving our decay Lemma, we recall  a regularity result from \cite{AM2} for Sobolev minimizers of autonomous variational integrals (see Theorem 3.2 therein), which will be used in our proof. 

\begin{proposition}\label{prop:nondip}
Let $f:\mathbb{R}^d\to\mathbb{R}$ be a continuous function satisfying 
\begin{equation}\label{eq:boundsf}
L^{-1}(\mu^2+|\xi^2|)^{p/2}\le f(\xi)\le L(\mu^2+|\xi|^2)^{p/2}
\end{equation}
for all $\xi\in\mathbb{R}^d$, where $0\le\mu\le 1$, $L\ge 1$, $1<p^-\le p\le p^+$, and
\begin{equation}\label{eq:decomposition}
\begin{split}
&\int_{Q}f(\xi+\nabla\varphi)\mathrm{d}x\ge \int_{Q}[f(\xi)+L^{-1}(\mu^2+|\xi|^2+|\nabla \varphi|^2)^{(p-2)/2}|\nabla\varphi|^2]\,\mathrm{d}x,
\end{split}
\end{equation}
for all $\xi\in\mathbb{R}^d$, $\varphi\in C_0^1(Q)$, $Q$ be the unit cube. 
Let $B_R$ be a ball in $\mathbb{R}^d$ and let $v\in W^{1,p}(B_R)$ be a local minimizer of the functional $w\mapsto\int_{B_R}f(\nabla w)\,\mathrm{d}x$. Then there exists a constant $C_0=C_0(p^-,p^+,L)$ such that
\begin{equation}\label{Linfty}
\sup_{B_{R/2}}(\mu^2+|\nabla v|^2)^{p/2}\le C_0\-int_{B_R}(\mu^2+|\nabla v|^2)^{p/2}\mathrm{d}x\,.
\end{equation}
\end{proposition}

\begin{remark}\label{rem:nondip}
Observe that the previous Theorem can be applied to
\[
f(\xi)=(\mu^2 + |\xi|^2)^{p/2}+h(\xi)
\] 
where $h$ is a convex function satisfying $0\le h(\xi) \le C(\mu^2+ |\xi|^2)^{p/2}$. Indeed, by \cite[Proposition 2.2]{FFM} a continuous function $f$ satisfying \eqref{eq:boundsf} admits the above decomposition  if and only if  \eqref{eq:decomposition} holds.
\end{remark}

The following result is a slight generalization of \cite[Lemma 3.2]{FMT} to sequences of functions.
\begin{lemma}\label{lem:lemma3.2fmt}
Let $(g_j)_{j\in\N}$, $g_j:\R^{d}\to[0,\infty)$, be a sequence of quasi-convex functions, and let $(p_j)_{j\in\N}$, with $p_j\geq1$ for every $j\in \N$, be a bounded sequence. Assume that 
\begin{equation*}
0\leq g_j(\xi) \leq L (1+|\xi|^{p_j})\,,\quad \forall \xi\in \R^{d}\,,\,\, \forall j\in\N\,,
\end{equation*}
and let $(t_j)\subset(0,\infty)$ be a sequence such that $\lim_j t_j=+\infty$. Then, setting
\begin{equation*}
\hat{g}_j(\xi):=\frac{g_{j}(t_{j}\xi)}{t_{j}^{p_{j}}}\,,\quad \xi\in\R^d\,,\,\, j\in\N\,,
\end{equation*}
there exists a subsequence $(t_{j_k})$ such that $\hat{g}_{j_k}$ converge to a  quasi-convex function $g_\infty$ uniformly on compact subsets of $\R^{d}$. If, in addition, $p_j\to p$ for some $p\geq1$, then
\begin{equation*}
g_\infty(\xi) \leq L |\xi|^p\,,\quad \forall \xi\in\R^d\,.
\end{equation*}
\end{lemma}
\begin{proof}
The assertion follows in a standard way noticing that each $g_j$ is continuous and complying with the estimate (see, e.g., \cite[Lemma~5.2]{GIUSTI})
\begin{equation*}
|g_j(\xi_1)-g_j(\xi_2)| \leq C (1+|\xi_1|^{p_j -1}+|\xi_2|^{p_j -1}) |\xi_1-\xi_2|\,,\quad \forall \xi_1,\xi_2\in\R^d\,,\,\,\forall j\in\N\,,
\end{equation*}
where $C=C(L,\sup_{j}p_j)$. Then, the sequence $\hat{g}_j$ is uniformly bounded and uniformly equicontinuous in any ball.  
\end{proof}
%

We are now in a position to state and prove the announced decay property.

\begin{lemma}[Decay estimate]\label{lem:decay}
Let $f$ be a function satisfying \eqref{effe}, \eqref{ipo1h}, \eqref{ipo2h}, for a strongly log-H\"older continous exponent $p(\cdot)$ complying with \eqref{eq:assump+}. There is a constant $C_1=C_1(d,L,p^-,p^+)$ with the property that for every $\tau\in (0,1)$ there exist $\varepsilon=\varepsilon(\tau)$, $\theta=\theta(\tau)$ in $(0,1)$ such that if $u\in SBV(\Omega)$ satisfies, for $x\in\Omega$ and $B_\rho(x)\ssubset\Omega$, $\rho<\varepsilon^2$,
\[
F(u,B_\rho(x))\le\varepsilon\rho^{d-1}, \ \ \ \ \ {\rm Dev}(u,B_\rho(x))\le\theta F(u,B_\rho(x)),
\]
then 
\begin{equation}\label{decade}
F(u,B_{\tau\rho}(x))\le C_1\tau^d F(u, B_\rho(x)).
\end{equation}
\end{lemma}

\begin{proof}
It is enough to assume $\tau\in (0,1/2)$ (otherwise just take $C_1=2^d$). {We argue by contradiction and assume that \eqref{decade} does not hold. }
In this case, there exist a sequence $u_j\in SBV(\Omega)$, sequences of nonnegative numbers $\varepsilon_j$, $\theta_j$, $\rho_j$, with $\lim_j \varepsilon_j=\lim_j\theta_j=0$, $\rho_j\le \varepsilon_j^2$, and $x_j\in\Omega$, with $B_{\rho_j}(x_j)\ssubset\Omega$, such that
\begin{equation}\label{assu1}
F(u_j,B_{\rho_j}(x_j))\le\varepsilon_j\rho_j^{d-1}\,, \ \ \ \ \ {\rm Dev}(u_j,B_{\rho_j}(x_j))\le\theta_j F(u_j,B_{\rho_j}(x_j))\,,
\end{equation}
and 
\begin{equation}\label{assu2}
F(u_j,B_{\tau\rho_j}(x_j))> C_1\tau^d F(u_j, B_{\rho_j}(x_j))\,,
\end{equation}
where $C_1=(1+L)C_0$ and $C_0$ comes from  \eqref{Linfty} in Proposition \ref{prop:nondip}. 

For every $j$, we consider the exponent $\bar{p}_j:=p(x_j)$, the scaled variable exponent $p_j: B_1\to(1,+\infty)$ and the function $v_j:B_1\to\mathbb{R}$ defined as
\begin{equation*}
p_j(y):=p(x_j+\rho_jy)
\end{equation*}
and
\[
v_j(y):=(\rho_j\gamma_j)^{1/\bar{p}_j}u_j(x_j+\rho_jy)\rho_j^{-1}\,,
\]
respectively, with $\gamma_j:=1/\varepsilon_j$. 
Now, if we set 
\[
F_j(v,\gamma_j,B_\rho):=\int_{B_\rho}f_j(x,\nabla v)\,\mathrm{d}y+\gamma_j\mathcal{H}^{d-1}(S_v,B_\rho)
\]
with $f_j(y,\xi):=\gamma_j\rho_jf(x_j+\rho_jy,(\gamma_j\rho_j)^{-1/\bar{p}_j}\xi)$, \eqref{assu1} and \eqref{assu2} can be rewritten respectively as
\begin{equation}\label{assv1}
F_j(v_j,\gamma_j,B_1)\le 1\,, \ \ \ \ \ {\rm Dev}_j(v_j,\gamma_j,B_1)\le\theta_j\,,
\end{equation}
and 
\begin{equation}\label{assv2}
F_j(v_j,\gamma_j,B_\tau)> C_1\tau^d F_j(v_j,\gamma_j, B_1)\,.
\end{equation}
The first bound in \eqref{assv1} in turn implies 
\[
\mathcal{H}^{d-1}(S_{v_j}\cap B_1)\le\varepsilon_j\,,
\]
and
\[
\int_{B_1}|\nabla v_j|^{p_j(y)}\,\mathrm{d}y\le \gamma_j\rho_j\int_{B_1}|\nabla u_j(x_j+\rho_jy)|^{p_j(y)}\,\mathrm{d}y\le 1\,,
\]
where we used that $\gamma_j\rho_j\le 1$.

Extracting eventually a further subsequence (not relabeled for convenience), we may also assume that $x_j\to x_0$, 
with $x_0\in\Omega$, so that, setting ${\bar p}:=p(x_0)$, we have
\begin{equation*}
\sup_{y\in B_1}|p_j(y)-\bar{p}| \leq \sup_{y\in B_1}|p(x_j+\rho_jy)-p(x_j)| + |p(x_j)- p(x_0)| \leq \omega(\rho_j) + o(1) \to0
\end{equation*}
as $j\to+\infty$. Thus, $p_j(y)\to {\bar p}$ uniformly 
in $B_1$. 
Taking into account Theorem \ref{thm:3.5} and Remark \ref{pcost}, we find a function $v_0\in W^{1,{\bar p}}(B_1)$ such that the convergences stated in \eqref{eq:claims} hold.

We now prove that the function $f_j(y, \xi)$ converges to a convex function $f_\infty(\xi)$ 
uniformly on compact subsets of $B_1\times\mathbb{R}^d$. Moreover, $|\xi|^{\bar p}\le f_\infty(\xi)\le L(1+|\xi|^{\bar p})$. 
{Setting
\begin{equation}
{\tilde f}_j(\xi):= \gamma_j\rho_jf(x_j, (\gamma_j\rho_j)^{-1/\bar{p}_j}\xi)\,,
\label{eq:scaledftilde}
\end{equation}
we start by proving that, if $y\in B_1$, $|\xi|\le R$, 
\[
|f_j(y,\xi)-{\tilde f}_j(\xi)|\le \omega_{j,R}
\]
for some 
$\omega_{j,R}$ which is infinitesimal as $j\to+\infty$. } In fact, we observe that
\[
\begin{split}
|f_j(y,\xi)-{\tilde f}_j(\xi)|&\le \gamma_j\rho_j\left|{|\xi|^{p_j(y)}}{(\gamma_j\rho_j)^{-\frac{p_j(y)}{\bar{p}_j}}}-{|\xi|^{\bar{p}_j}}{(\gamma_j\rho_j)^{-1}}\right|\\
& \,\,\,\,\,\, +\gamma_j\rho_j\left|h(x_j+\rho_jy,(\gamma_j\rho_j)^{-1/\bar{p}_j}\xi)-h(x_j, (\gamma_j\rho_j)^{-1/\bar{p}_j}\xi)\right| \\
& =: {\mathcal I}_j+ {\mathcal J}_j\,.
\end{split}
\]
We first estimate $\mathcal {I}_j$: by triangle inequality, we have
\[
{\mathcal I}_j\le \left(\frac{1}{\gamma_j\rho_j}\right)^{\frac{p_j(y)}{\bar{p}_j}-1}\left| |\xi|^{p_j(y)}-|\xi|^{\bar{ p}_j}\right|+|\xi|^{\bar{p}_j}\left|\left(\frac{1}{\gamma_j\rho_j}\right)^{\frac{p_j(y)}{\bar{p}_j}-1}-1\right|=:{\mathcal I}_{j,1}+{\mathcal I}_{j,2}\,.
\]
As for the coefficient of ${\mathcal I}_{j,1}$, we note that
\begin{equation*}\label{limi1}
\left|\left(\frac{1}{\gamma_j\rho_j}\right)^{\frac{p_j(y)-{\bar{p}_j}}{\bar{p}_j}}-1\right|\le \mathrm{e}^{\frac{|p_j(y)-{\bar{p}_j}|\log\left(\frac{1}{\gamma_j\rho_j}\right)}{\bar{p}_j}}-1\le 
\mathrm{e}^{-\frac{1}{p^-}\omega(\gamma_j\rho_j)\log(\gamma_j\rho_j)}-1\,,
\end{equation*}
whence
\begin{equation}
\lim_{j\to+\infty}\left(\frac{1}{\gamma_j\rho_j}\right)^{\frac{{p}_j(y)}{\bar{p}_j}-1}=1,
\label{limi1bis}
\end{equation}
since $\lim_j(\gamma_j\rho_j)=0$ and $\omega$ satisfies \eqref{eq:rinfologhold}. On the other hand it is easy to prove that $|\xi|^{p_j(y)}-|\xi|^{\bar{p}_j}$ converges to $0$ uniformly on $B_1$ if $|\xi|\le R$. We now treat ${\mathcal I}_{j,2}$: we have
\[
{\mathcal I}_{j,2}\le R^{\bar{p}_j}
\left|\left(\frac{1}{\gamma_j\rho_j}\right)^{\frac{p_j(y)-{\bar{p}_j}}{\bar{p}_j}}-1\right| \leq \max\{R^{p^-}, R^{p^+}\}\left|\left(\frac{1}{\gamma_j\rho_j}\right)^{\frac{p_j(y)-{\bar{p}_j}}{\bar{p}_j}}-1\right|\,,
\]
and the term in the right hand side is infinitesimal as $j\to+\infty$ by virtue of \eqref{limi1bis}. 
We are only left with estimating $\mathcal{J}_j$: 
by \eqref{ipo2h} we have 
\[
\begin{split}
\mathcal{J}_j&\le L\,\gamma_j\rho_j\,\omega(\rho_j)\left[|\xi|^{p_j(y)}(\gamma_j\rho_j)^{-\frac{p_j(y)}{\bar{p}_j}}+|\xi|^{\bar{p}_j}(\gamma_j\rho_j)^{-1}\right]\left[1+|\log(|\xi|(\gamma_j\rho_j)^{-\frac{1}{\bar{p}_j}})\right]\\
&\le c\,L\,\max\{R^{p^-}, R^{p^+}\}\log|R|\,\omega(\gamma_j\rho_j)\left[1+\log\left(\frac{1}{\gamma_j\rho_j}\right)\right]\left[\left(\frac{1}{\gamma_j\rho_j}\right)^{\frac{p_j(y)}{\bar{p}_j}-1}+1\right]\,,
\end{split}
\]
the last term being bounded again by \eqref{limi1bis}. 
Therefore,
\[
\mathcal{J}_j\le c\, \max\{R^{p^-}, R^{p^+}\}\log|R|\,\omega(\gamma_j\rho_j)\left[1+\log\left(\frac{1}{\gamma_j\rho_j}\right)\right]=:\omega_{j,R}.
\]
Now, thanks to {Lemma~\ref{lem:lemma3.2fmt}, applied to the sequences $p_j:=\bar{p}_j$, $g_j(\xi):=h(x_j,\xi)$ with $t_j:=(\gamma_j\rho_j)^{-1/\bar{p}_j}$, since ${\tilde f}_j(\xi)=|\xi|^{\bar{p}_j}+\hat{g}_j(\xi)$ } 
we may conclude that up to another subsequence 
\[
f_j(y,\xi)\to |\xi|^{\bar{p}}+g_\infty(\xi) 
\]
uniformly on $B_1\times K$, where $K$ is a compact subset in $\mathbb{R}^d$, for a convex function $g_\infty$ such that $0\le g_\infty(\xi)\le L|\xi|^{\bar p}$.  Now, set $f_\infty(\xi):=|\xi|^{\bar{p}}+g_\infty(\xi)$.  
Thus, Theorem \ref{gammaconv} allows us to conclude that $v_0$ is a local minimizer of the functional $v\mapsto\int_{B_1}f_\infty(\nabla v)\,\mathrm{d}x$ and 
\[
\lim_{j\to+\infty}F_j(v_j,\gamma_j,B_\rho)=\int_{B_\rho}f_\infty(\nabla v_0)\,\mathrm{d}x \ \ \ \ \forall \rho\in (0,1).
\]
Let us note that thanks to Proposition \ref{prop:nondip} and Remark \ref{rem:nondip},  $v_0$ satisfies
\[
\sup_{B_\tau}|\nabla v_0|^{\bar p}\le C_0\-int_{B_1}|\nabla v_0|^{\bar p}\,\mathrm{d}y\le C_0\-int_{B_1}f_\infty(\nabla v_0)\,\mathrm{d}y\,.
\]
In conclusion 
\[
\begin{split}
\lim_{j\to+\infty}F_j(v_j,\gamma_j,B_\tau)&=\int_{B_\tau}f_\infty(\nabla v_0)\,\mathrm{d}y\le (1+L)\sup_{B_\tau}|\nabla v_0|^{\bar p}\tau^d \mathcal{L}^d(B_1)\\
&\le (1+L)C_0\tau^d\int_{B_1}f_\infty(\nabla v_0)\,\mathrm{d}y\le (1+L)C_0\tau^d\limsup_{j\to+\infty}F_j(v_j,\gamma_j,B_1)\,,
\end{split}
\]
which provides the contradiction to \eqref{assv2}.

\end{proof}

\subsection{Density lower bound} \label{sec:densitylowerb}

In order to study the regularity of the jump set $S_u$, a key tool will be an Ahlfors-type regularity result, ensuring that $F(u,B_\rho(x))$, where $B_\rho(x)$ is any ball centred at a jump point $x\in S_u$, is controlled from above and from below. 

We first recall the definition of \emph{quasi-minimizer} (see \cite[Definition~7.17]{Ambrosio-Fusco-Pallara:2000}).

\begin{definition}
A function $u\in SBV_{\rm loc}(\Omega)$ is a quasi-minimizer of the functional $F(v,\Omega)$ if there exists a constant $\kappa\geq0$ such that for all balls $B_\rho(x)\ssubset\Omega$ it holds that
\begin{equation}
{\rm Dev}(u, B_\rho(x)) \leq \kappa \rho^d\,.
\label{eq:devquasim}
\end{equation}
{The class of quasi-minimizers complying with \eqref{eq:devquasim} is denoted by $\mathcal{M}_\kappa(\Omega)$.}
\end{definition}

{It is easy to check (see, e.g., \cite[Remark~7.16]{Ambrosio-Fusco-Pallara:2000}) that any minimizer $u$ of the functional $\mathcal{F}$  in \eqref{eq:functionals}  belongs to $\mathcal{M}_\kappa(\Omega)$ with $\kappa:= 2^q \alpha \gamma_d \|g\|_\infty^q$. }

The following upper bound is quite immediate, as it follows from a standard comparison argument. Note that here the assumption $x\in S_u$ is, actually, not needed.
\begin{lemma}[Energy upper bound] Assume that $f$ complies with \eqref{effe}, and {let $u\in \mathcal{M}_\kappa(\Omega)$.} 
Then, for all balls $B_\rho(x)\subset\Omega$
\begin{equation}
\int_{B_\rho(x)} f(y,\nabla u)\,\mathrm{d}y + \mathcal{H}^{d-1}(S_u\cap B_\rho(x)) \leq d\gamma_d\rho^{d-1} + {\kappa'\rho^d},
\label{eq:7.23AFP}
\end{equation}
where $\kappa'=\kappa+L\gamma_d$.
\end{lemma}
\begin{proof}
See, e.g., \cite[Lemma~7.19]{Ambrosio-Fusco-Pallara:2000}, with $\kappa$ replaced by $\kappa'$ since $f(x,0)\le L$ in $B_\rho(x)$.
\end{proof}

On the contrary, the following lower bound for $F(u,B_\rho(x))$ requires that the small balls $B_\rho(x)$ be centred at $x\in {S_u}$. The proof is based on the decay estimate of Lemma~\ref{lem:decay}, and it follows along the lines of the proof of \cite[Theorem~7.21]{Ambrosio-Fusco-Pallara:2000} where $p$ is constant. The difference is in the introduction of the set $\Sigma$ where the $L^1_{\rm loc}$-function $|\nabla u|^{p(\cdot)}$ is locally large (see \eqref{eq:setsigma} below), which turns out to be $\mathcal{H}^{d-1}$-negligible. For the sake of completeness, we prefer to provide all the details. 

\begin{theorem}[Density lower bound]\label{thm:mainthm} Let $f$ be a function satisfying \eqref{effe}, \eqref{ipo1h}, \eqref{ipo2h}, for a strongly log-H\"older continous exponent $p$ complying with \eqref{eq:assump+}. There exist $\theta_0$ and $\rho_0$ depending on $d, p^-,p^+$ and $L$ with the property that if $u\in SBV(\Omega)$ is a {quasi-minimizer} of $F$ in $\Omega$, then 
\begin{equation}
F(u, B_\rho(x))>\theta_0\rho^{d-1}
\label{eq:7.24AFP}
\end{equation}
for all balls $B_\rho(x)\subset\Omega$ with centre $x\in\overline{S}_u$ and radius $\rho<\frac{\rho_0}{\kappa'}$, where $\kappa'$ is that of \eqref{eq:7.23AFP}. Moreover, 
\begin{equation}
\mathcal{H}^{d-1}((\overline{S}_u\setminus S_u)\cap \Omega)=0.
\label{eq:7.24AFPbis}
\end{equation}
\end{theorem}
\begin{proof}
Without loss of generality, we may assume that $x=0$. Let $0<\tau<1$ be fixed such that $\sqrt{\tau}\leq \frac{1}{C_1}$ and set $\varepsilon_0:=\varepsilon(\tau)$, where $C_1$ and $\varepsilon(\tau)$ are given from Lemma~\ref{lem:decay}. Let $0<\sigma<1$ be such that 
\begin{equation}
\sigma \leq \frac{\varepsilon_0}{C_1 (d\gamma_d+1)}\,,
\label{eq:choicesigma}
\end{equation}
and set
\begin{equation}
\rho_0:= \min\{1,\varepsilon(\sigma)^2,\varepsilon_0\tau^d\theta(\tau), \varepsilon_0\sigma^{d-1}\theta(\sigma)\}\,,
\label{eq:choicerho0}
\end{equation}
where $\theta(\tau)$ and $\theta(\sigma)$ are the constants of Lemma~\ref{lem:decay} corresponding to $\tau$ and $\sigma$, respectively.
First, we claim that if $\rho<\frac{\rho_0}{\kappa}$ and $B_\rho\subset\Omega$, then
\begin{equation}
F(u, B_\rho) \leq \varepsilon(\sigma)\rho^{d-1}
\label{eq:7.25AFP}
\end{equation}
implies
\begin{equation}
F(u, B_{\sigma\tau^m\rho}) \leq \varepsilon_0 \tau^\frac{m}{2}(\sigma\tau^m\rho)^{d-1}\,,\quad \forall m\in\N \,.
\label{eq:7.26AFP}
\end{equation}
We first prove \eqref{eq:7.26AFP} for $m=0$. With \eqref{eq:7.25AFP}, \eqref{eq:choicerho0} and if 
\begin{equation}
{\rm Dev}(u,B_{\rho})\leq \theta(\sigma)F(u,B_\rho)
\label{eq:7.27afp}
\end{equation} 
holds, in view of \eqref{decade}, \eqref{eq:7.23AFP} and the choice of $\sigma$ \eqref{eq:choicesigma} we deduce that
\begin{equation}
\begin{split}
F(u,B_{\sigma\rho}) \leq C_1\sigma^d F(u,B_\rho) & \leq  C_1\sigma^d (d\gamma_d\rho^{d-1}+\kappa'\rho^d) \\
 &\leq (\sigma\rho)^{d-1} C_1 \sigma (d\gamma_d + 1) \\
 & \leq \varepsilon_0 (\sigma\rho)^{d-1}\,,
\end{split}
\end{equation}
as desired. {If \eqref{eq:7.27afp} does not hold, then from the definition of $\rho_0$ and the quasi-minimality of $u$ we infer
\begin{equation*}
F(u, B_{\sigma\rho}) \leq F(u, B_{\rho}) \leq \frac{1}{\theta(\sigma)} {\rm Dev}(u,B_{\rho})  \leq \frac{\kappa\rho^d}{\theta(\sigma)} \leq \varepsilon_0 (\sigma\rho)^{d-1}\,,
\end{equation*}
proving again \eqref{eq:7.26AFP} for $m=0$. }

Now, we assume that \eqref{eq:7.26AFP} is valid for some $m\geq0$, and that 
\begin{equation}
{\rm Dev}(u,B_{\sigma\tau^m\rho})\leq \theta(\tau)F(u,B_{\sigma\tau^m\rho})\,.
\label{eq:7.28} 
\end{equation}
Then, by Lemma~\ref{lem:decay} and the choice of $\tau$ we get
\begin{equation*}
\begin{split}
F(u, B_{\sigma\tau^{m+1}\rho}) & \leq C_1\tau^d F(u, B_{\sigma\tau^m\rho}) \\
  & \leq C_1\tau^d\varepsilon_0 \tau^\frac{m}{2}(\sigma\tau^m\rho)^{d-1} \\
  & \leq \varepsilon_0 \tau^\frac{(m+1)}{2}(\sigma\tau^{m+1}\rho)^{d-1}\,,
\end{split}
\end{equation*}
which corresponds to \eqref{eq:7.26AFP} for $m+1$. If \eqref{eq:7.28} is not true, using \eqref{eq:devquasim}, 
we get
\begin{equation*}
\begin{split}
F(u, B_{\sigma\tau^{m+1}\rho}) \leq F(u, B_{\sigma\tau^{m}\rho}) \leq \frac{1}{\theta(\tau)} {\rm Dev}(u,B_{\sigma\tau^m\rho}) & \leq \frac{\kappa}{\theta(\tau)}(\sigma\tau^m\rho)^d \\
& \leq \varepsilon_0 \tau^{\frac{m+1}{2}}(\sigma\tau^{m+1}\rho)^{d-1}\,.
\end{split}
\end{equation*}
This proves the claim. \\
\\
Now, we assume \eqref{eq:7.25AFP} for some ball $B_\rho\subset\Omega$, with $\rho<\frac{\rho_0}{\kappa}$. From \eqref{eq:7.26AFP} we get
\begin{equation*}
\lim_{r\to0} \frac{F(u, B_r)}{r^{d-1}}=0\,,
\end{equation*}
whence, {by using the inequality $|\xi|^{p^-}\leq 1+|\xi|^{p(\cdot)}$, we infer
\begin{equation*}
\lim_{r\to0} \frac{1}{r^{d-1}} \int_{B_r(x)}|\nabla u|^{p^-}\,\mathrm{d}y=0\,.
\end{equation*}} Therefore, {Theorem~\ref{thm:thm7.8AFP} with $p=p^-$} and $q=1^*$ implies that $0\in I$, where
\[
I=\left\{x\in\Omega : \limsup_{r\to 0}\-int_{B_r(x)}|u(y)|^{1^*}\,\mathrm{d}y=+\infty\right\}.
\]
Then \eqref{eq:7.24AFP} holds true for all $x\in S_u\setminus I$. Exploiting the uniformity of the bound, by a density argument, the inequality is still true for balls centred in  $x\in \overline{S_u\setminus I}$. We are left to prove that $\overline{S_u\setminus I}=\overline{S}_u$.
Let $x\not\in\overline{S_u\setminus I}$; we first observe that the set $I$ is $\mathcal{H}^{d-1}$-negligible (see \cite[Lemma 3.75]{Ambrosio-Fusco-Pallara:2000}. Thus, we can find a neighborhood $V$ of $x$ such that $\mathcal{H}^{d-1}(V\cap S_u)=0$, and this implies in a standard way that $u\in W^{1,p^-}(V)$. Now, by virtue of the classical Poincar\'e inequality for Sobolev functions and the upper bound \eqref{eq:7.23AFP} we get
\begin{equation}
\dashint_{B_\rho(x)}  \left|\frac{u-(u)_{x,\rho}}{\rho^\alpha}\right|^{p^-}\,\mathrm{d}y \leq c \,,
\label{eq:campanatoexcess}
\end{equation}
where $(u)_{x,\rho}$ denotes the average of $u$ in $B_\rho(x)$ and $\alpha:=(p^--1)/p^-$.
With \eqref{eq:campanatoexcess} and the Campanato's characterization of H\"older continuity (see, e.g., \cite{Camp1}), we infer that a representative of $u$ belongs to $C^{0,\alpha}(V)$. Therefore, $x\not\in \overline{S}_u$, and this concludes the proof of \eqref{eq:7.24AFP}.

{As for \eqref{eq:7.24AFPbis}, it follows from \eqref{eq:7.24AFP} by a geometric measure theory argument. Let us define the set 
\begin{equation}
\Sigma:=\left\{x\in\Omega:\,\,\mathop{\lim\sup}_{r\to0} \frac{1}{r^{d-1}}\int_{B_r(x)}f(y,\nabla u)\,\mathrm{d}y>0\right\}\,.
\label{eq:setsigma}
\end{equation}
Since $|\nabla u|^{p(\cdot)}\in L^1_{\rm{loc}}(\Omega)$, it holds that $\mathcal{H}^{d-1}(\Sigma)=0$  (see, e.g., \cite[Section~2.4.3, Theorem~3]{EG}). 
On the other hand, if $x\in \Omega\cap(\overline{S}_u\setminus\Sigma)$, from \eqref{eq:7.24AFP} we derive that the density 
$$
\Theta_{d-1}(S_u,x):=\limsup_{r\to 0}\frac{\mathcal{H}^{d-1}(S_u\cap B_r(x))}{d\gamma_dr^{d-1}}\ge\theta_0>0\,.
$$
Let us define, for any Borel set $E\subset \mathbb{R}^d$ , the Radon measure $\mu(E) = \mathcal{H}^{d-1}(E\cap S_u)$. 
Since from $\Theta_{d-1}(E,x)\ge \theta_0$ for every $x\in B$, $B$ a Borel set, we infer  that $\mu\ge \theta_0\,\mathcal{H}^{d-1}\ristretto B$ (see, e.g., \cite[Theorem~2.56]{Ambrosio-Fusco-Pallara:2000}), choosing $B=\Omega\cap(\overline{S}_u\setminus S_u)\setminus\Sigma$, we deduce that  $\mu(\Omega\cap(\overline{S}_u\setminus S_u)\setminus\Sigma)\ge\theta_0\,\mathcal{H}^{d-1}(\Omega\cap(\overline{S}_u\setminus S_u)\setminus\Sigma)$. Thus we get $\mathcal{H}^{d-1}(\Omega\cap(\overline{S}_u\setminus S_u)\setminus\Sigma)=0$, which in turn implies \eqref{eq:7.24AFPbis}, since $\mathcal{H}^{d-1}(\Sigma)=0$.}

\end{proof}

We are now in position to prove the following existence result for minimizers of $\mathcal{F}$ defined in \eqref{eq:functionals}. Since the integrand $f$ does not depend on $u$, it is immediate to check that $\mathcal{F}(u)$ is non increasing by truncations. Therefore, in order to minimize $\mathcal{F}$, we may restrict to those $u\in SBV(\Omega)$ such that $\|u\|_\infty\leq M$, where $M:=\|g\|_\infty$. 

\begin{theorem}\label{thm:main}
Let $f$ comply with the assumptions of Theorem~\ref{thm:mainthm}. Then there exists a minimizer $u\in SBV(\Omega)\cap L^\infty(\Omega)$ of $\mathcal{F}$ defined in \eqref{eq:functionals}. Moreover, the pair $(\overline{S}_u,u)$ is a minimizer of the functional $\mathcal{G}$ \eqref{eq:functionalsG}.
\end{theorem}

\begin{proof}
As $f(x, \cdot)$ is convex for every $x\in \Omega$, and $p(\cdot)$ is superlinear, the existence of a minimizer $u$ as above can be established in a standard way by using the lower semicontinuity result of \cite{DeG} combined with the classical closure and compactness results in $SBV$ (see \cite[Theorem~4.7 and 4.8]{Ambrosio-Fusco-Pallara:2000}). Using the minimality property of $u$, a comparison with the constant function $v=0$ gives $u\in W^{1,p(\cdot)}(\Omega\setminus\overline{S}_u)$. Moreover, \eqref{eq:7.24AFPbis} holds. Now, let $(K,v)$ be any competitor with $\mathcal{G}(K,v)<+\infty$. Since the functionals are non increasing by truncations, it is not restrictive to assume $v\in L^\infty(\Omega)$, whence (see \cite[Theorem~4.4]{Ambrosio-Fusco-Pallara:2000}) $v\in SBV(\Omega)$ and $\mathcal{H}^{d-1}(S_v\setminus K)=0$. Using the minimality of $u$ again, we conclude that $\mathcal{G}(\overline{S}_u,u)=\mathcal{F}(u)\leq \mathcal{F}(v)\le \mathcal{G}(K,v)$.
\end{proof}

\section*{Acknowledgements} 

The authors are members of Gruppo Nazionale per l'Analisi Matematica, la Probabilit\`a e le loro Applicazioni (GNAMPA) of INdAM.
G. Scilla and F. Solombrino have been supported by the project STAR PLUS 2020 – Linea 1 (21‐UNINA‐EPIG‐172) ``New perspectives in the Variational modeling of Continuum Mechanics''. The work of F. Solombrino is part of the project “Variational methods for stationary and evolution problems with singularities and interfaces” PRIN 2017 financed by the Italian Ministry of Education, University, and Research.

\bibliographystyle{siam}

\end{document}